\documentclass[11pt,a4paper, reqno]{article}

\usepackage{indentfirst} 
\usepackage[utf8]{inputenc}
\usepackage{amsmath, amsthm, amssymb,xfrac, mathrsfs} 
\usepackage{tikz} 
\usetikzlibrary{hobby}
\usetikzlibrary{decorations.pathreplacing}
\usepackage{wrapfig} 
\usepackage{xcolor} 
\usepackage{verbatim} 
\usepackage{enumerate} 

\usepackage{geometry}
\usepackage{hyperref} 
\hypersetup{
colorlinks=true,
citecolor=black!50!red,
linkcolor=black!50!red,
linktoc=all
}
\usepackage[all]{hypcap} 

\newcounter{constant}

\newtheorem{theorem}{Theorem}[section]
\newtheorem{proposition}[theorem]{Proposition}
\newtheorem{lemma}[theorem]{Lemma}

\newtheorem{cor}[theorem]{Corollary}

\numberwithin{equation}{section} 

\theoremstyle{definition}
\newtheorem{definition}[theorem]{Definition}

\newtheorem{remark}[theorem]{Remark}

\renewcommand{\P}{\mathbb{P}}
\newcommand{\E}{\mathbb{E}}
\newcommand{\R}{\mathbb{R}}

\newcommand{\N}{\mathbb{N}}

\newcommand{\PP}{\textnormal{P}}

\newcommand{\EE}{\textnormal{E}}

\DeclareMathOperator{\dist}{d}
\renewcommand{\d}{\textnormal{d}}

\DeclareMathOperator{\poisson}{Poisson}

\DeclareMathOperator{\binomial}{Binomial}

\newcommand{\pnorm}[2]{\left\|#2\right\|_{#1}}
\DeclareMathOperator{\diam}{diam}

\begin{document}

\title{Large deviations for marked sparse random graphs with applications to interacting diffusions}

\date{\today}
\author{Rangel Baldasso
  \thanks{E-mail: \ rangel@puc-rio.br; \ Department of Mathematics, PUC-Rio, Rua Marqu\^{e}s de S\~{a}o Vicente 225, G\'{a}vea, 22451-900 Rio de Janeiro, RJ - Brazil.}
  \and
  Roberto Oliveira
  \thanks{E-mail: \ rimfo@impa.br; \ IMPA, Estrada Dona Castorina 110, 22460-320 Rio de Janeiro, RJ - Brazil.}
  \and
  Alan Pereira
  \thanks{E-mail: \ alan.pereira@im.ufal.br, alan.anderson.math@gmail.com; \ Instituto de Matem\'{a}tica, Universidade Federal de Alagoas, Rua Lorival de Melo Mota s/n, 57072970 Macei\'{o}, AL - Brazil.}
  \and
  Guilherme Reis
  \thanks{E-mail: \ ghpreis@id.uff.br; \ Instituto de Matemática e Estatística, Departamento de Matemática Aplicada (GMA),  Rua Professor Marcos Waldemar de Freitas, s/n
Blocos G e H - Campus do Gragoatá - São Domingos
24.210-201 Niterói - RJ - Brazil.}}
\maketitle

\begin{abstract}
We consider the empirical neighborhood distribution of marked sparse Erd\H{o}s-R\'enyi random graphs, obtained by decorating edges and vertices of a sparse Erd\H{o}s-R\'enyi random graph with i.i.d.\ random elements taking values on Polish spaces.
We prove that the empirical neighborhood distribution of this model satisfies a large deviation principle in the framework of local weak convergence.
We rely on the concept of BC-entropy  introduced by Delgosha and Anantharam~(2019) which is inspired on the previous work by Bordenave and Caputo~(2015).
Our main technical contribution is an approximation result that allows one to pass from graph with marks in discrete spaces to marks in general Polish spaces.
As an application of the results developed here, we prove a large deviation principle for interacting diffusions driven by gradient evolution and defined on top of sparse Erd\H{o}s-R\'enyi random graphs.
In particular, our results apply for the stochastic Kuramoto model.
We obtain analogous results for the sparse uniform random graph with given number of edges.

\medskip

\noindent
\emph{Keywords and phrases:} large deviations, local topology,  sparse random graphs, interacting diffusions, stochastic Kuramoto model.

\noindent
MSC 2010: 60F10, 05C80.
\end{abstract}

\section{Introduction}\label{s:intro}

\par Random graphs appear naturally when one tries to model complex networks and connections within individuals of a community. Many times, the interest lies in properties of a given process that has these random graphs as underlying environment. One usually needs to understand finer properties of the random graph model in order to successfully achieve a deeper understanding of the random process itself. Here, our goal is to further develop tools that allow one to obtain large deviations results for processes defined on top of random graphs.

The main motivation of this work is to understand the particular case of the stochastic Kuramoto model with interaction structure given by random graphs. In order to define this model fix a finite time horizon $T >0$, a (possibly random) sequence of growing graphs $G_{n} = ([n], E_{n})$\footnote{We use the standard notation $[n]=\{1, 2, \dots, n\}$.}, and let $(B_i : i \in [n])$ be a collection of independent standard Brownian motions. The stochastic Kuramoto model is the collection of diffusions $(\theta_i;\,i\in [n])$ that solves the system of It\^o stochastic differential equations
\begin{equation}
  \d \theta_i(t)=\sum_{j:j\sim_{n} i}\sin(\theta_i(t)-\theta_j(t)) \,\d t+\d B_i(t), \qquad t \in [0,T] \text{ and } i \in [n],
\end{equation}
where $j \sim_{n} i$ in the index of the summation above denotes adjacency in the graph $G_{n}$.

Our interest lies in considering the sequence of graphs $G_{n}$ as sparse Erd\H{o}s-R\'enyi random graph with average degree $d>0$. In this context, the quantity of interest is the empirical distribution of the solutions
\begin{equation}\label{eq:empiricalm}
  L_n(\vec{\theta})=\frac{1}{n}\sum_{i=1}^n\delta_{\theta_i},
\end{equation}
which can be seen as a probability measure on the set of bounded continuous functions defined on the interval $[0,T]$.

In the framework of large deviations, Dai Pra and Den Hollander~\cite{dpdh} consider such systems in the mean field case, when the interaction graph $G_n$ is the complete graph. The authors combine Girsanov's Theorem and It\^o's formula to derive that
\begin{equation}\label{eq:RN}
\begin{gathered}
\text{the Radon-Nykodim derivative of } L_n(\vec{\theta}) \text{ w.r.t.\ }L_n(\vec{B})
\\
\text{ can be written as } \exp(F(L_n)),
\end{gathered}
\end{equation}
where $F$ is a suitably chosen bounded continuous function of the empirical distribution $L_{n}(\vec{B})$. Sanov's Theorem and Varadhan's Lemma allow them to conclude from~\eqref{eq:RN} that the empirical measure in~\eqref{eq:empiricalm} satisfies a large deviation principle with an explicit rate function.

To the best of our knowledge, our work is the first to pursue large deviation results for the case when the sequence of interaction graphs are locally tree-like. We show in Theorem~\ref{thm:intdiff} that the strategy in~\eqref{eq:RN} works if one manages to replace the use of Sanov's Theorem by a large deviation principle for the following ``generalized'' empirical measure
\begin{equation}\label{eq:empiricalmGn} 
\frac{1}{n}\sum_{i=1}^n \delta_{\{ (B_i,B_j,B_k) : j,k \sim_{n} i \}}.
\end{equation}

Observe that, even though the collection $(B_i)_i$ is i.i.d., the empirical measure above also takes into account the graph structure of $G_n$, preventing any attempts of applying Sanov's Theorem.

\bigskip

\noindent \textbf{Main results.}
Motivated by the expressions in~\eqref{eq:empiricalmGn}, we here focus on the study of large deviations for empirical distributions of random marked graphs.

In order to state our results, let us first fix some notation. Let $\Theta$ and $\Xi$ be two Polish spaces and consider two fixed distributions on $\nu$ and $\chi$ on $\Theta$ and $\Xi \times \Xi$, respectively. We construct a random marked graph by first sampling a random graph $G_{n}$ distributed according to a sparse Erd\H{o}s R\'enyi random graph $G \big(n, \tfrac{d}{n} \big)$ or as a uniformly chosen graph with $n$ vertices and $m$ edges. Afterwards independently attach to each vertex and edge a ``mark'' with distribution $\nu$ and $\chi$, respectively. Besides, since each edge receives two marks in $\Xi$, we also randomly associate each of these marks with an orientation of the edge, that is, we choose one orientation to receive the first mark and the opposite orientation receives the other mark. This marked random graph will be denoted by $\bar{G}_{n}$.

For each $i \in [n]$, denote by $\bar{G}_{n}(i)$ the connected component of the marked graph $\bar{G}_{n}$ that contains the vertex $i$ and by $[\bar{G}_{n}(i),i]$ the marked graph $\bar{G}_{n}(i)$ rooted at $i$. We consider the empirical distribution
\begin{equation}\label{eq:empirical_distribution_ER}
U(\bar{G}_{n})=\frac{1}{n}\sum_{i=1}^n\delta_{[\bar{G}_{n}(i),i]},
\end{equation}
which is a (random) probability distribution on the set of rooted connected marked graphs (up to isomorphism) $\bar{\mathcal{G}}_{*}$.

Our first result states the large deviation principle for the uniform marked graph with i.i.d.\ marks.
\begin{theorem}\label{t:ldp_uniformg_general}
  Consider two Polish spaces $\Theta$ and $\Xi$.  
  Let $\nu \in \mathcal{P}(\Theta)$, $\chi\in\mathcal{P}(\Xi \times \Xi)$, and consider $G_n$ the marked graph with given number of edges and i.i.d.\ marks (distributed according to $\nu$ and $\chi$) and underlying graph uniformly sampled from $\mathcal{G}_{n,m_n}$ (the set of graphs with $n$ vertices and $m_{n}$ edges) with $\tfrac{m_{n}}{n} \to \tfrac{d}{2}$.
  The sequence of empirical neighborhood measures $U(G_n)$ satisfies a large deviation principle with good rate function $\bar{I}^{u}_d(\mu)$ given via~\eqref{eq:liminfliminf}.
\end{theorem}

When one considers the case of the sparse Erd\H{o}s R\'enyi random graph, from the point of view of local weak convergence, the marked random graph $\bar{G}_{n}$ converges to the Galton-Watson tree with degree distribution $\poisson(d)$ and i.i.d.\ marks on vertices and edges with respective distributions $\nu$ and $\chi$. Our main result concerns the large deviations behavior around this convergence, that can be expressed via the empirical distributions above.
\begin{theorem}\label{t:ldp_erg_general}
Fix $\nu \in \mathcal{P}(\Theta)$ and $\chi\in\mathcal{P}(\Xi\times\Xi)$. Let $G_n$ be the marked graph with i.i.d. discrete marks (with distributions $\nu$ and $\chi)$ and underlying graph the Erd\H{o}s-R\'enyi random graph $G(n,p)$ with $p=\tfrac{d}{n}$.
The sequence of empirical neighborhood measures $U(G_n)$ satisfies a large deviation principle with a good rate function $\bar{I}_d^{\rm ER}$.
\end{theorem}

We are also able to obtain expressions for the rate functions. If the spaces $\Theta$ and $\Xi$ are finite, we prove via mixture techniques that the function $I$ can be explicitly written in terms of the BC-entropy introduced in Delgosha and Anantharam~\cite{da}. If on the other hand $\Theta$ and $\Xi$ are general Polish metric spaces, suitable discretizations of these spaces are used in order to obtain expressions for the rate function.

The empirical measure~\eqref{eq:empirical_distribution_ER} (which contains more information than~\eqref{eq:empiricalmGn})  is used to prove a large deviation profile for the stochastic Kuramoto model in the sparse regime via Theorem~\ref{t:ldp_erg_general} and Varadhan's Lemma. See Theorem~\ref{thm:intdiff}.

A natural question that arises from this result is if the phase transition for the existence of the giant component can be recovered from the rate function. Just like in~\cite{bc}, this is not the case, due to the fact that the topology we consider here only gives local information about the graphs. 

\begin{remark}
We characterize one minimizer of the rate functions $\bar{I}^{u}_d(\mu)$ and $\bar{I}_{d}^{\rm ER}$ in Proposition~\ref{prop:minimizer_general}.
\end{remark}

\bigskip

\subsection{Proof overview}
\noindent \textbf{BC-entropy.} Assume for a moment that the graph does not have any marks. Denote by $\mathcal{G}_{n,m}$ the space of graphs with $n$ vertices and $m$ edges. Later on we allow $m$ to depend on $n$ and thus omit it from the notation. Let $G_n$ be uniformly sampled from $\mathcal{G}_{n,m}$. For each vertex $i\in [n]$ let $[G_n(i),i]$ be its connected component rooted at $i$.  The empirical neighborhood distribution of $G_n$ is the random measure
\begin{equation}\label{eq:endg}
U(G_n)=\frac{1}{n}\sum_{i=1}^n\delta_{[G_n(i),i]},
\end{equation}
taking values on the space of probability measures defined on the space $\mathcal{G}_*$ of rooted connected graphs (up to isomorphism). In this context, Bordenave and Caputo~\cite{bc} establish sharp asymptotic exponential estimates for probabilities of the form
\begin{equation}\label{eq:prob=frac}
  \P\big[U(G_n)\in B(\rho,\delta)\big]=\frac{|\{G\in \mathcal{G}_n\,:\, U(G)\in B(\rho,\delta)\}|}{|\mathcal{G}_n|},
\end{equation}
where $\rho$ is a probability measure over rooted graphs, $\delta$ is any positive number, and $B(\rho,\delta)$ is the ball in the (metrizable) weak topology, thus obtaining a large deviation principle for such empirical measures.

The BC-entropy, introduced by Delgosha and Anantharam~\cite{da}, generalizes and plays a role similar to the rate function in the large deviation principle established in~\cite{bc}, but in the context of marked graphs. In this setting, through a reconstruction result~\cite[Proposition~8]{da}, the authors are able to use a generalized configuration model to understand asymptotics of~\eqref{eq:prob=frac} for marked graphs.

In order to have a rough statement of the results in~\cite{da}, fix two finite spaces $\Theta$ and $\Xi$, which represent the mark spaces for vertices and oriented edges, respectively. Let $G$ be a marked graph with $n$ vertices and $m$ edges and define, for each $\theta \in \Theta$, $u_{G}(\theta)$ as the number of vertices of $G$ with mark $\theta$. Analogously, define $m_{G}(x,x')$ as the number of edges that receive marks $x,x' \in \Xi$, regardless of orientation.

Fix sequences of vectors $\vec{u}_n=(u_n(\theta);\,\theta \in \Theta)$ and $\vec{m}_{n}= (m_n(x,x');\,x,x' \in \Xi)$. Assume that $\vec{u}_n/n$ converges to a probability measure $Q$ on $\Theta$ and $\vec{m}_{n}/n$ converges to a vector of non-negative mean degrees $\vec{d}=(d_{x,x'}; x,x' \in \Xi)$. Let $\mathcal{G}_{\vec{u}_{n}, \vec{m}_{n}}^{\rm da}$ denote the collection of marked graphs $G$ with $n$ vertices and such that $\vec{u}_{G} = \vec{u}_{n}$ and $\vec{m}_{G} = \vec{m}_{n}$, and consider a random graph $G^{\rm da}_{n}$ uniformly sampled from $\mathcal{G}_{\vec{u}_{n}, \vec{m}_{n}}^{\rm da}$.

The following theorem, proved in~\cite{da}, is the starting point of our results.
\begin{theorem}[\cite{da}, see Theorem~\ref{t:equality_entropy} for a complete statement]\label{t:A}
The sequence $U(G^{\rm da}_{n})$ satisfies a weak large deviation principle with a rate function that depends only on the pair $(\vec{d},Q)$.
\end{theorem}

From the result above, we can first conclude Theorem~\ref{t:ldp_uniformg_general} for the case when $\Theta$ and $\Xi$ are discrete spaces. This is obtained by constructing these marked graphs as mixtures of $G^{\rm da}_{n}$ when the vectors $\vec{u}_n$ and $\vec{m}_{n}$ are chosen with appropriate distributions.

With this result we handle the case when the random graph is distributed according to the sparse Erd\H{o}s-R\'enyi random graph $G \big(n, \tfrac{d}{n} \big)$. This follows from the mixture structure of the distribution of the Erd\H{o}s-R\'enyi random graph. This is the idea of the proof of the large deviations result for the case when the spaces $\Theta$ and $\Xi$ are finite. 

\bigskip

\noindent \textbf{The case of general Polish spaces.} We now drop the assumption that the spaces $\Theta$ and $\Xi$ are finite. To go from discrete to a general Polish space we rely on suitable discretizations of the spaces of marks. The general large deviation result is deduced from ``approximate'' large deviation estimates.

To illustrate the above, let $\bar{G}_n$ be a random marked graph. Consider marked graphs $\bar{G}_{n}^{\pi}$ whose marks are obtained by applying certain projection maps $\pi: \Theta \to \tilde{\Theta}$ and $\tilde{\pi}: \Xi \to \tilde{\Xi}$ to the marks of the graph $\bar{G}_{n}$. Assume that we already know that the empirical neighborhood measure of $\bar{G}_{n}^{\pi}$ satisfies a large deviation principle if the spaces $\tilde{\Theta}$ and $\tilde{\Xi}$ are finite. By choosing the sequence of suitable discretizations $(\pi_{k}, \Theta^{k})_{k \in \N}$ and $(\tilde{\pi}_{k}, \Xi^{k})_{k \in \N}$, we conclude that the empirical neighborhood measure of $\bar{G}_{n}$ satisfies a large deviation principle as well. The main technical contribution of the current paper is the following lemma, used to prove that marked random graphs can be well approximated by graphs with marks in finite spaces.

In what follows, we denote by $[G,v]_{h}$ for the ball of radius $h$ centered at $v \in G$ rooted at $v$. See Subsection~\ref{subsec:UGW} for more details.
\begin{lemma}\label{lemma:distace_estimate}
  Assume that $\Theta$ and $\Xi$ are Polish spaces, and consider a pair $(\mathcal{A}, \mathcal{B})$ of $(\epsilon, \delta)$-good tagged partitions \footnote{See Section~\ref{sec:partproject} for the definition}  of $\Theta$ and $\Xi$ with respect to the measures $\nu$ and $\chi$. Let $G$ be any graph with marks in the spaces $\Theta$ and $\Xi$ and let $G^\pi$ be the respective marked graph with projected marks.
  Then, there exists sets $A^{*} \subset \Theta$ and $B^{*} \in \Xi \times \Xi$ such that, for any given parameters $h\geq 1$ and $S \geq 1$,
\begin{equation}\label{eq:bound_BL}
\begin{split}
\d_{BL}\left(U(G),U(G^{\pi})\right) \leq & \frac{1}{1+h} + \epsilon + \frac{2}{n}\sum_{v\in V(G)}\mathbf{1}_{\{|[G,v]_{2h}|>S\}}\\
& + 2S\bigg(\frac{1}{n}\sum_{v \in V(G)}\mathbf{1}_{\{\tau_{v} \in A^{*}\}} + \frac{1}{n}\sum_{e\in E(G)}\mathbf{1}_{\{\xi_{e} \in B^{*}\}}\bigg),
\end{split}
\end{equation}
where $\d_{BL}$ denotes the bounded Lipschitz distance (see~\eqref{eq:BL_norm}).
\end{lemma}

\begin{remark}
With the notation of Section~\eqref{sec:partproject}, we can take $A^{*}= A_{\ell}$ and $B^{*} = ( B_{\ell'} \times \Xi ) \cup ( \Xi \times B_{\ell'} )$.
\end{remark}

In fact, we prove that, provided these discretizations are constructed in a particular way, the empirical measures of the model with marks in $(\Theta, \Xi)$ and its discrete counterpart are exponentially close in the sense of Lemma~\ref{lema:Kurtz}.

If the sets $(\Theta, \Xi)$ are compact metric spaces, we discretize them by simply taking finite coverings by balls with small radius, and approximate marks by the center of the ball to which they belong.
As we shall see, this produces marked random graphs whose empirical distributions are close according to the bounded Lipschitz metric.

If we assume that the spaces $(\Theta, \Xi)$ are not compact, we first need to restrict our probability measure to compact sets with big mass and then perform a similar analysis. In this case, it is necessary to take into account an extra term when bounding the distance between the empirical distributions, that comes from the possibility that marks fall outside this fixed compact set.

We now give a rough idea of the proof of Lemma~\ref{lemma:distace_estimate}. For fixed $h\in \N$, the distance between $U(G)$ and $U(G^{\pi})$ is bounded by the distance between $U(G)_h$ and $U(G^{\pi})_h$.
  Since the underlying graphs of $G$ and $G^{\pi}$ are equal, one only needs to bound the distance between marks in the same edge or vertex.
  This distance will be small for all vertices and edges whose marks do not fall in $A^{*}$ or $B^{*}$, respectively.
  It remains then to count the number of bad marks we have, i.e., marks belonging to either of the sets $A^{*}$ or $B^{*}$.
  In the empirical measure $U(G)_h$, one bad mark can influence many neighborhoods $[G,v]_h$, and the control on this amount will be obtained by bounding the number of bad marks together with a bound on the number of neighborhoods each edge and vertex can influence.

\bigskip

\subsection{Related works}
Many works deal with unmarked graphs, where large deviations have been studied under different aspects. Chatterjee and Varadhan~\cite{cv} study large deviations for Erd\"{o}s-R\'{e}nyi random graphs in the dense regime when $p$ remains constant. In this case, the large deviation result is framed in the setting of the cut topology introduced by Lov\'{a}sz and Szegedy~\cite{ls}.

For the sparse regime, the framework of local weak convergence introduced by Benjamini and Schramm~\cite{bs} and Aldous and Steele~\cite{as} is more useful, since a strikingly large collection of graph functionals is continuous under this topology. In this sense, studying large deviations for empirical neighborhoods allows one to obtain similar estimates for such functionals. Large deviation estimates for graphs in this regime were first treated by Bordenave and Caputo~\cite{bc}, specifically for the Erd\H{o}s-R\'{e}nyi random graph and configuration model. More recently, Backhausz, Bordenave, and Szegedy~\cite{bbs} consider the case of uniformly sampled $d$-regular random graphs.

When one turns to marked graphs, the picture is more scarce. Based on~\cite{bc}, Delgosha and Anantharam~\cite{da} introduce the entropy notion that is central in our work and deduce a weak large deviation principle for the uniform marked graph with given mark counting vectors (see Section~\ref{sec:da_results}).

Andreis, K\"onig, and Patterson~\cite{ak} derive a large deviation principle for the empirical measure of the sizes of all the connected components of the Erd\"{o}s-R\'{e}nyi in the sparse regime. Andreis, K\"onig, Langhammer, and Patterson~\cite{ak2} prove a large deviation principle for the statistics of the collection of all the connected components when the Erd\"{o}s-R\'{e}nyi is marked and inhomogeneous.

In comparison to our result, similar models are explored in~\cite{ak2} where the authors consider the number of vertices of a specific type (mark) within each connected component of the random graph. Notably, their focus lies on obtaining detailed information about the sizes of connected components and offering alternative proofs for classical results related to the phase transition concerning the existence of a giant component. However, a remarkable difference emerges as their work disregards information about the vertex neighborhoods. In contrast, our research places significance on the edge connections within a ball of finite radius around a uniformly sampled vertex. Our results revolve around the geometry of these edge connections in conjunction with the marks of the rooted ball. It is essential to highlight that our scope is limited to finite-radius balls, resulting in the loss of information about the entire connected component. 	Significantly, our findings do not appear to directly imply the results presented in~\cite{ak2}, nor do we perceive a straightforward implication in the reverse direction. The nature of our exploration, focusing on finite-radius balls and the associated edge connections, diverges from the considerations in~\cite{ak2}. We posit that further investigation is warranted to fully understand and delineate the potential relationships between these distinct yet interconnected works.

\bigskip

Let us mention works regarding interacting diffusions defined on top of  random graphs. Dai Pra and den Hollander~\cite{dpdh} consider gradient dynamics on the complete graph. This was weakened by Baldasso, Pereira, and Reis~\cite{bpr} to the non-gradient case and with the presence of delays.

When the underlying structure is given by the Erd\"{o}s-R\'{e}nyi random graph, Delattre, Giacomin, and Lu\c{c}on~\cite{Delattre2016} compare the system with the respective mean-field case and establish bounds on the distances between the solutions, provided the mean degree diverges with logarithmic speed. In the dense regime, without any assumptions on the velocity of divergence of the mean degree, but with more restricted interactions than in~\cite{Delattre2016}, Oliveira and Reis~\cite{or} obtain large deviation estimates for the solution paths. We refer the reader to~\cite{or} to have a more detailed discussion of works in the dense regime.

 In the sparse setting, Lacker, Ramanan and Wu~\cite{Kavita2019} and Oliveira, Reis and Stolerman~\cite{OliveiraReisStolerman2018} deduce convergence of the solutions to the properly defined model on the Galton-Watson random tree. However, their results do not address large deviations.

MacLaurin~\cite{2016maclaurin} considers an interacting particle system defined through stochastic differential equations in networks that converge in the local weak sense to $\mathbb{Z}^d$, fitting the sparse setting. They obtain large deviation principle by applying a series of transformations to an already known large deviation principle (see \cite[Section~4]{2016maclaurin}). The methods in this paper seem to be suitable only for networks that converge in the local weak sense to $\mathbb{Z}^d$. In particular, the above results do not consider other networks as Erd\"{o}s-R\'{e}nyi or the uniform graph in the sparse regime.

\bigskip

\noindent \textbf{Acknowledgments.} 
RB is supported by Conselho Nacional de Desenvolvimento Científico e Tecnológico CNPq grants {\em Produtividade em Pesquisa} (308018/2022-2), {\em Universal} (402952/2023-5), and the Mathematical Institute of Leiden University for the hospitality during the elaboration of this work.
RIO gratefully acknowledges support from grants {\em Produtividade em Pesquisa} (304475/2019-0 and 305765/2023-0) and {\em Universal} grants (432310/2018-5 and 402952/2023-5) from CNPq, Brazil; an {\em Excellent Science - Marie Skłodowska-Curie Actions} grant from the European Commission
(DOI: 10.3030/101007705); and the following grants from FAPERJ, Rio de Janeiro,
Brazil: {\em Cientista do Nosso Estado} (E-26/200.485/2023) and {\em Edital Inteligência
Artificial} (E-26/290.024/2021).
AP was funded by Fundação de Amparo à Pesquisa do Estado de Alagoas - FAPEAL, Brazil (E:60030.0000002397/2022) and the support of CNPq {\em Universal} grant (402952/2023-5).
GR was partially supported by  a Capes/PNPD fellowship 888887.313738/2019-00 while he was a post doc at Federal University of Bahia (UFBA).
The authors thank IMPA for the hospitality during the elaboration of this work and Charles Bordenave for fruitful discussions.

\section{Notation}\label{app:notation}

\subsection{Graphs and marks}

\par Given a graph $G$, we denote by $V(G)$ and $E(G)$ the vertex and edge sets of $G$, respectively.
  For two given vertices $u,v \in V(G)$, we write $u \sim v$ if there exists an edge between these two vertices.
  The degree of $u$ in $G$, denoted by $\deg_G(u)$, is the number of vertices $v\in G$ such that $v \sim u$.
  For connected graphs, the distance between two given vertices $u$ and $v$ is the quantity $\dist_{G}(u,v)$ defined as the size of a minimal path connecting $u$ and $v$.
  The graph $G'=(V',E')$ is an induced subgraph of $G=(V,E)$ if $V' \subset V$ and $\{u,v\} \in E$ if, and only if, $\{u,v\} \in E'$, for all $u,v \in V'$.

  A rooted graph $(G,o)$ is a graph with a distinguished vertex $o \in V(G)$.
  An isomorphism between two rooted graphs $(G,o)$ and $(G',o')$ is a bijection between the vertex sets $V(G)$ and $V(G')$ that preserves the edge sets and the roots.
  For $h \in \N$, denote by $(G,o)_{h}$ as the ball (with respect to the metric $\dist_{G}$) of radius $h$ in $G$ with center at the root $o$.

  We consider only graphs with finite or countably infinite vertex sets.
  In the latter case, we assume that it is locally finite, meaning that each vertex has finite degree.
  We also assume all graphs are simple: no loops and no multiple edges are allowed.


  Roughly speaking, a marked graph is a graph $G=(V,E)$ with marks associated to the vertices and oriented edges of $G$.
  We denote by $\vec{E}$ the collection of oriented edges of $G$
\begin{equation}
\vec{E} = \big\{ \vec{e}=(u,v): \{u, v\} \in E \big\}.
\end{equation}
  Let $(\Theta,{\rm d}_{\Theta})$ and $(\Xi,{\rm d}_{\Xi})$ be two Polish metric spaces.
  A marked graph $\bar{G}=(G,\vec{\tau},\vec{\xi})$ is a graph $G=(V,E)$ together with the fields of marks
\begin{equation}\label{eq:fields-mark}
\vec{\tau}=(\tau(v))_{v\in V} \quad \text{and} \quad \vec{\xi}=(\xi(\vec{e} \, ))_{\vec{e} \in \vec{E}},
\end{equation}
where $\tau(v) \in \Theta$, for all $v \in V$, and $\xi(\vec{e} \, ) \in \Xi$, for all $\vec{e} \in \vec{E}$.
  We denote by $\bar{\mathcal{G}}=\bar{\mathcal{G}}_{(\Theta,\Xi)}$ the space of connected and locally finite marked graphs with mark spaces $\Theta$ and $\Xi$.

  Given a marked graph $\bar{G}=(G,\vec{\tau},\vec{\xi} \,)$, we say that $\bar{G}'=(G',\vec{\tau}',\vec{\xi}' \, )$ is induced by $G'=(V',E')$ if $G'$ is an induced subgraph of $G$, the field $\vec{\tau}'$ is the restriction of $\vec{\tau}$ to $V'$, and $\vec{\xi}'$ is the restriction of $\vec{\xi}$ to $E'$.

  When we write a graph property for a marked graph $\bar{G}=(G,\vec{\tau},\vec{\xi})$ it is implicitly assumed that this property holds for the underlying graph.
  For example, $V(\bar{G}):=V(G)$, $E(\bar{G}):=E(G)$, and, for a vertex $v \in V(\bar{G})$, $\deg_{\bar{G}}(v)=\deg_G(v)$.
  A rooted marked graph $(\bar{G},o)$ is a marked graph $\bar{G}$ with a distinguished vertex $o$.
  We might sometimes also denote a marked graph by $\bar{G}=(V, E, \vec{\tau}, \vec{\xi} \,)$, where $(V,E)$ denotes the underlying graph.

\subsection{Isomorphisms}\label{sec:iso}

\par An isomorphism between two rooted marked graphs $(\bar{G},o)=(G, \vec{\tau},\vec{\xi}, o)$ and $(\bar{G}',o)=(G', \vec{\tau}',\vec{\xi}', o')$ is a bijection $\Psi: V \to V'$ that preserves edges, marks, and the root. In particular, an isomorphism between two rooted marked graphs induces an isomorphism between the underlying rooted graphs $(G,o)$ and $(G',o')$.

  Recall that, for a graph $G$ and a radius $r\in \N$, $(G,o)_{r}$ denotes the subgraph of $G$ induced by the vertices within distance $r$ from $o \in V(G)$. Analogously, given a rooted marked graph $(\bar{G},o)$, we denote by $(\bar{G},o)_r$ the marked graph induced by $(G,o)_r$ rooted at the vertex $o$.

\bigskip

  Let $\bar{\mathcal{G}}_{*}$ denote the space of (connected and locally finite) rooted marked graphs with mark spaces $(\Theta, \Xi)$ up to isomorphism.
  We often identify classes with representative elements without further mention to this.
  Furthermore, we write $\bar{\mathcal{T}}_* \subset \bar{\mathcal{G}}_{*}$ for the space of (connected and locally finite) rooted marked trees up to isomorphism.
  Also, let $\bar{\mathcal{G}}^{h}_{*}$ denote the set of rooted marked connected and locally finite graphs with depth at most $h$, meaning that all vertices are within distance $h$ from the root.
  We write $\mathcal{G}_{*}$ for the space of connected locally finite rooted graphs (without marks).
  There exists a natural projection $\pi_g: \bar{\mathcal{G}}_{*} \to \mathcal{G}_{*}$ that associates to each marked graph $\bar G$ its graph component $\pi_g(\bar G)$.

  Finally, let $\mathcal{G}_n$ denote the set of graphs and $\bar{\mathcal{G}_n}$ the set of marked graphs with vertex set $[n]$ (notice that these sets include graphs that are not connected).

\subsection{Local weak convergence}\label{sec:lwc}

\par In order to define local weak convergence of graphs, we first introduce a metric structure on the space $\bar{\mathcal{G}}_{*}$.
  This is not the exact same notion as in~\cite{bordenave}, but it is equivalent to it, as a simple calculation shows.

  Consider two rooted marked graphs $(\bar{G},o)=(G,\vec{\tau},\vec{\xi},o)$ and $(\bar{G}',o')=(G',\vec{\tau}',\vec{\xi}',o')$ belonging to $\bar{\mathcal{G}}_*$.
  Given $r\in \N$ and $\delta>0$, we say that the pair $(r,\delta)$ is good for $((\bar{G},o),(\bar{G}',o'))$ if there exists an isomorphism $\Psi$ between the two underlying graphs of $(\bar{G},o)_r$ and $(\bar{G}',o')_r$ such that the corresponding marks are at most $\delta$ apart: 
\begin{enumerate}
\item ${\rm d}_\Theta(\tau_v,\tau'_{\Psi(v)})< \delta$, for all $\ v \in (\bar{G},o)_r$,
\item ${\rm d}_\Xi\left(\xi_{(u,v)},\xi'_{(\Psi(u),\Psi(v))}\right)< \delta$, for all $\{u,v\}\in (\bar{G},o)_r$.
\end{enumerate}
  The distance between $[\bar{G},o]$ and $[\bar{G}',o']$ is defined by
\begin{equation}\label{def:distN}
\dist_{\bar{\mathcal{G}}_*}([\bar{G},o],[\bar{G}',o']) = \inf\left\{\frac{1}{1+r} + \delta\,:\,(r,\delta)\mbox{ is good for } ((\bar{G},o),(\bar{G}',o')) \right\}.
\end{equation}

  One can verify that $(\bar{\mathcal{G}}_*,{\rm d}_{\bar{\mathcal{G}}_*})$ is a Polish space (cf.~\cite{bordenave}).

\bigskip

  For a (not necessarily connected) marked graph $\bar{G}$ and a vertex $v \in V$, we associate the rooted marked graph $(\bar{G}(v),v)$ given by the marked graph induced by the connected component of $v$ rooted at $v$.
  We write $[\bar{G}(v),v] \in \bar{\mathcal{G}}_{*}$ for the equivalence class of the rooted connected marked graph $(\bar{G}(v),v)$.

  For a finite marked graph $\bar{G}$ with vertex set $V$, we define the \emph{empirical neighborhood distribution} as
\begin{equation}\label{eq:empirical_neighborhood_distribution}
U(\bar{G})=\frac{1}{|V|}\sum_{v\in V}\delta_{[\bar{G}(v),v]}.
\end{equation}
  Notice that, even though we do not require $\bar{G}$ to be connected, $U(\bar{G})$ is a distribution in $\bar{\mathcal{G}}_{*}$.

  Denote by $\mathcal{P}\left(\bar{\mathcal{G}}_{*}\right)$ the collection of probability measures on $\bar{\mathcal{G}}_{*}$ endowed with the weak convergence topology.
  This space can be metrized by the L\'{e}vy-Prokhorov metric $\d_{LP}$, defined as follows.
  Denote by $\mathcal{B}(\bar{\mathcal{G}}_{*})$ the Borel $\sigma$-algebra of the metric space $(\bar{\mathcal{G}}_*,{\rm d}_{\bar{\mathcal{G}}_*})$.
  For each $A \in \mathcal{B}(\bar{\mathcal{G}}_{*})$ let 
\begin{equation*}
A^{\varepsilon} = \{ g \in \bar{\mathcal{G}}_{*}: {\d}_{\bar{\mathcal{G}}_*}(g,g') <\varepsilon, \text{ for some } g' \in A \}.
\end{equation*}
For $\mu, \nu \in \mathcal{P}\left(\bar{\mathcal{G}}_{*}\right)$, define their L\'{e}vy-Prokhorov distance as
\begin{equation*}
\d_{LP}(\mu ,\nu )= \inf \Big\{
\varepsilon >0:  
  \begin{array}{l}
   \mu (A)\leq \nu (A^{\varepsilon})+\varepsilon \text{ and} \\ 
   \nu (A)\leq \mu (A^{\varepsilon})+\varepsilon, \text{ for all } A \in \mathcal{B}(\bar{\mathcal{G}}_{*})
  \end{array}
\Big\}.
\end{equation*}

Another metric compatible with the weak topology is the bounded Lipschitz metric
\begin{equation}
\d_{\rm BL}(\mu, \nu) = \sup \left \{ \left| \int f \, \d \nu - \int f \, \d \mu \right|: f \in BL(M) \right\},
\end{equation}
where $BL(M)$ is the class of $1$-Lipschitz functions $f : M \to \mathbb{R}$ bounded by one.

For the next lemma, let $x \wedge y$ denote the minimum between $x$ and $y$.
\begin{lemma}
Let $(X, Y)$ be a coupling of two distributions $\mu$ and $\nu$. Then
\begin{equation}
\d_{\rm BL}(\mu,\nu)\leq \mathbb{E}(\d(X,Y) \wedge 2).
\end{equation}
\end{lemma}

\begin{definition}[Local weak convergence]\label{def:localweak}
Consider a sequence of marked graphs $\bar{G}_n=([n],E_n,\vec{\tau}_n,\vec{\xi}_n)$ and let $\rho \in \mathcal{P}\left(\bar{\mathcal{G}}_{*}\right)$. We say that $\bar{G}_n$ converges locally weakly to $\rho$ if $U(\bar{G}_n)$ converges to $\rho$ in the sense of weak convergence.
\end{definition}

In the remaining of the paper we will drop the bar from $\bar G$ and write $G$ for a marked graph.

\par In the case when $\Theta$ is finite, we define, for $\mu \in \mathcal{P}(\bar{\mathcal{G}}_*)$, and $\theta \in \Theta$,
\begin{equation*}
\Pi_\theta (\mu) := \mu( \tau_G(o) = \theta ) \mbox{ and } \vec{\Pi}(\mu):= (\Pi_\theta(\mu): \theta \in \Theta).
\end{equation*}
Also, for $x,x' \in \Xi$, let $\deg^{x,x'}_G(o)$ denote the number of edges adjacent to the origin with marks $x$ and $x'$. Furthermore, set
\begin{equation}\label{eq:mean-degrees}
\begin{split}
\deg^{x,x'}(\mu) &:= \E_{\mu}(\deg^{x,x'}_G(o)),\\
\quad \deg(\mu) &:= \sum_{x,x'} \deg^{x,x'}(\mu), \text{ and}\\
\vec{\deg}(\mu) &:= (\deg^{x,x'}(\mu): x, x' \in \Xi).
\end{split}
\end{equation}

\subsection{Unimodularity}\label{subsec:unimodularity}

\par Let $\bar{\mathcal{G}}_{**}$ denote the set of connected marked graphs with two distinguished vertices up to isomorphisms that preserve both vertices.
  The local topology of $\bar{\mathcal{G}}_{*}$ extends in a natural way to $\bar{\mathcal{G}}_{**}$.
 
  A measure $\mu \in \mathcal{P}(\bar{\mathcal{G}}_*)$ is called \emph{unimodular} if, for any non-negative measurable function $f: \bar{\mathcal{G}}_{**} \rightarrow \R_+$, we have
\begin{equation}\label{eq:unimodularity}
\int \sum_{v\in V (G)} f([G,o,v]) \, \d\mu([G,o]) =
\int \sum_{v\in V (G)} f([G,v,o]) \, \d\mu([G,o]).
\end{equation}
  We denote the set of unimodular probability measures on $\bar{\mathcal{G}}_*$ by $\mathcal{P}_u (\bar{\mathcal{G}}_*)$.

  We refrain from discussing the properties of unimodular distributions in depth and refer the reader to~\cite{as}.
  A proposition that will be important for us is the following.
\begin{proposition}
The set $\mathcal{P}_u (\bar{\mathcal{G}}_*)$ is closed on $\mathcal{P}(\bar{\mathcal{G}}_*)$.
\end{proposition}

\subsection{Marked unimodular Galton-Watson trees}\label{subsec:UGW}

\par In this subsection, we review the concept of marked unimodular Galton-Watson trees introduced in~\cite{da}. We start by recalling the idea of depth-$h$ type of a marked graph, and finish with the definition of admissibility and of marked unimodular Galton-Watson tree.

\bigskip

\noindent{\bf Depth-$h$ type.} The depth $h$-type of marked graph encodes its local structure and is crucial in several results we mention in this work.

For a marked graph $G$ and adjacent vertices $u$ and $v$, define $G(u,v)=(\xi_G (u,v),(G' ,v))$, where $G'$ is the connected component of $v$ in the graph obtained from $G$ by removing the edge between $u$ and $v$. Similarly,
for an integer $h \geq 0$, $G(u,v)_h$ is defined as $(\xi_G (u,v),(G' ,v)_h )$. Let $G[u,v]$ denote the pair $(\xi_G (u,v),[G' ,v]) \in \Xi \times \bar{\mathcal{G}}_*$ and, for $h \geq 0$, let $G[u,v]_h$ denote $(\xi_G (u,v),[G' ,v]_h ) \in \Xi \times \bar{\mathcal{G}}^h_*$.

\par For a marked graph $G$, two adjacent vertices $u,v$ in $G$, and $h \geq 1$, the  \textit{depth-$h$ type} of the edge $(u,v)$ is defined as
\begin{equation}\label{eq:color}
\phi^h_G (u,v) := (G[v,u]_{h-1} ,G[u,v]_{h-1} ) \in (\Xi \times
\mathcal{G}^{h-1}_*
) \times (\Xi \times
\mathcal{G}^{h-1}_*
).
\end{equation}
For a given marked rooted graph $(G,o)$ and $g,g' \in \Xi\times \bar{\mathcal{G}}^{h-1}_*$, we define the quantity
\begin{equation}\label{eq:E_1}
E_h (g,g' )(G,o) := |\{v \sim_G o : \phi^h_G (o,v) = (g,g')\}|.
\end{equation}
Furthermore, for $\mu \in \mathcal{P}(\bar{\mathcal{G}}_{*})$ and $g,g' \in \Xi\times \bar{\mathcal{G}}^{h-1}_*$, set
\begin{equation}\label{eq:P}
e_{\mu} (g,g' ) := \E_{\mu}( E_h (g,g' )(G,o)).
\end{equation}
The first quantity above measures the number of edges adjacent to the root with depth-$h$ type $(g,g')$, which might be thought of as the degree of the depth-$h$ type $(g,g')$. Similarly, the second quantity is the average degree of depth-$h$ type $(g,g')$.

\bigskip

\noindent{\bf Admissibility.} Marked unimodular Galton-Watson trees are defined from a given probability measure $\PP \in \mathcal{P}(\bar{\mathcal{G}}^h_* )$, which we call the {\it seed} of the tree. The notion of admissibility gives us a necessary condition on the seed in order to generate such process, and is related to the unimodularity of the random tree constructed.

\begin{definition}\label{def:admissible}
Let $h \geq 1$. A probability distribution $\PP \in \mathcal{P}(\bar{\mathcal{G}}^h_* )$ is called admissible if $\E_{\PP}( \deg_G (o)) <\infty$ and $e_\PP (g,g' ) = e_\PP (g' ,g)$, for all $g,g' \in \Xi \times \bar{\mathcal{G}}^{h-1}_*$.
\end{definition}

\bigskip

\noindent{\bf Marked unimodular Galton-Watson trees.} Before defining the marked unimodular Galton-Watson trees we need some additional notation. For $g \in \Xi \times \bar{\mathcal{G}}_{*}$, we write $g=(g[m],g[s]),$ where $g[m]\in \Xi$ is the mark component of $g$ and $g[s]\in \bar{\mathcal{G}}_*$ is the subgraph component of $g$. For $t,t' \in \Xi\times \bar{\mathcal{T}}_*$ , define $t \oplus t' \in
\bar{\mathcal{T}}_*$ as the isomorphism class of the rooted tree $(T,o)$, where $o$ has a subtree isomorphic to $t[s]$ and an extra offspring $v$, whose subtree rooted at $v$ and without the edge $\{v, o\}$ is isomorphic to $t' [s]$. Furthermore, $\xi_T (v,o) = t[m]$ and $\xi_T (o,v) = t' [m]$. An illustration of this procedure can be found in~\cite[Figure~3]{da}.

For $h \geq 1$ and an admissible $\PP \in \mathcal{P}(\bar{\mathcal{T}}^h_*)$, we now introduce a collection of auxiliary probability measures in order to define our main object. For $t, t'\in \Xi \times \bar{\mathcal{T}}_*^{h-1}$ such that $e_\PP(t,t')>0$, define $\widehat{\PP}_{t,t'}\in \mathcal{P}(\Xi\times \bar{\mathcal{T}}_*^{h})$ via
\begin{equation*}
\widehat{\PP}_{t,t'}(\,\widetilde{t}\,):=\textbf{ 1 }_{\{\widetilde{t}_{h-1}=t\}}\frac{\PP(\widetilde{t}\oplus t')E_h(t,t')(\widetilde{t}\oplus t')}{e_\PP(t,t')}.
\end{equation*}
In case that $e_P(t,t')=0$, we define $\widehat{\PP}_{t,t'}(\,\widetilde{t}\,)=\textbf{1}_{\{\widetilde{t}=t\}}$. Admissiblity of $\PP$ implies that the distributions above are indeed probability measures (\cite[Section~2]{da}).

We will define the Borel probability measure  UGWT$_h (\PP) \in \mathcal{P}( \bar{\mathcal{T}}_* )$, which is called the marked unimodular Galton-Watson tree with depth-$h$ neighborhood distribution $\PP$. Recall that we also refer to the measure $\PP$ as the seed of UGWT$_h(\PP)$.

We construct $[T,o]$ with law UGWT$_h(\PP)$ recursively. First we sample $(T,o)_h$ according to $\PP$. Then, for each $v \sim_T o,$ we independently sample $\widetilde{t}\in \Xi\times \bar{\mathcal{T}}_*^h$ according to the law $\widehat{\PP}_{t,t'}(\,\cdot\,)$ where $t=T[o,v]_{h-1}$ and $t'=T[v,o]_{h-1}.$ Since $\widetilde{t}_{h-1}=t$, we can add one layer to $T(o,v)_{h-1}$ so that $T(o,v)_h=\widetilde{t}$. From this step, we obtain $(T,o)_{h+1}$. This procedure is then repeated for each vertex at subsequent depths, constructing $(T,o)$. A formal description of this distribution can be found in~\cite{da}.

Next proposition contains a summary of the main properties of UGWT$_h(\PP)$ that we will use throughout the text. For a given distribution $\mu \in \mathcal{P}(\bar{\mathcal{G}}_*)$, we denote by $\mu_h \in \mathcal{P}(\bar{\mathcal{G}}^h_*)$ the law of $[G,o]_h$ when $[G,o] \sim \mu$.

\begin{proposition}[\cite{da}, Section~2] Let $h\geq 1$ and fix $\PP \in \mathcal{P}(\bar{\mathcal{T}}_*^h)$ admissible.
\begin{enumerate}
\item The probability {\rm UGWT}$_h(\PP)\in \mathcal{P}(\bar{\mathcal{T}}_*)$ is well defined. Furthermore, this is a unimodular distribution and $\big({\rm UGWT}_h(\PP)\big)_{h}=\PP$.
\item (Finite-order Markov property) For all $k\geq h$,
\begin{equation*}
{\rm UGWT}_k\big(({\rm UGWT}_h(\PP))_k \big)={\rm UGWT}_h(\PP).
\end{equation*}
\end{enumerate}
\end{proposition}

\subsection{Review of BC entropy and weak large deviation principle}\label{sec:da_results}

\par We now review the needed notation and some of the main results from~\cite{da}, that introduced a cornerstone model in order to consider more general random marked graphs.  Their results are the starting point to the theorems we will prove throughout this work.

\noindent \textbf{Count vectors and degrees.}
In this entire section, we assume that the metric spaces $\Theta$ and $\Xi$ are finite. Whenever these spaces are finite we endow them with an order $\leq$ that we make use from now on.

Let $G$ be a finite marked graph. We define the edge-mark count vector of $G$ by
$\vec{m}_G := (m_G (x,x') \,:\, x,x'\in \Xi)$, where $m_G (x,x')$ is the number of \emph{oriented edges} with associated marks $x$ and $x'$, i.e., number of oriented edges $(v,w)$ such that $\xi_{G}(v,w) = x$ and $\xi_{G}(w,v)=x'$  or $\xi_{G}(v,w) = x'$ and $\xi_{G}(w,v)=x$. Analogously, the vertex-mark count vector of $G$ is $\vec{u}_G := (u_G (\theta) : \theta \in \Theta)$, where $u_G (\theta)$ is the number of vertices $v \in V (G)$ with $\tau_G (v) = \theta$.

\par For a vertex $o \in V (G)$, $\deg^{x,x'}_G(o)$ denotes the number of vertices $v$ connected to $o$ in $G$ such that $\xi_G (v,o) = x$ and $\xi_G (o,v) = x'$. Notice that $\deg_G (o)$ is precisely $\sum_{x,x'\in \Xi} \deg^{x,x'}_G(o)$.

\noindent \textbf{The DA model.}
Let us now introduce the model considered in~\cite{da}, that will also be the main	object considered throughout this section.

\par Fix a sequence of edge-mark count vectors and a sequence of vertex-mark count vectors
\begin{align*}
\vec{u}^{(n)}&=(u^{(n)}(\theta):\theta \in \Theta),\\
\vec{m}^{(n)}&=(m^{(n)}(x,x'):x,x'\in \Xi) \mbox{ with } m^{(n)}(x,x')=m^{(n)}(x',x).
\end{align*}
It will be convenient to consider the restriction of the matrix $\vec{m}^{(n)}$ to the lower diagonal of $\Xi^2$:
\begin{equation}\label{def:lowerd}
\vec{m}_\leq^{(n)}=(m^{(n)}(x,x'):x,x'\in \Xi \mbox{ with } x\leq x').
\end{equation}
We define
\begin{equation}\label{def:norm1}
\big\|\vec m^{(n)} \big\|_{1} := \frac{1}{2}\sum_{x \neq x'} m^{(n)}(x,x')+\sum_{x} m^{(n)}(x,x) = \sum_{x \leq x'}m^{(n)}(x,x'),
\end{equation}
and $\pnorm{1}{\vec u} = \sum_{\theta \in \Theta} u^{(n)}(\theta)$.

\par For each $n\in \N$, define
\begin{equation*}
\mathcal{G}^{(n)}_{\vec m,\vec u}=\{G : V(G)=[n],\,\vec{u}_G=\vec u^{(n)}, \, \vec{m}_G=\vec m^{(n)}\},
\end{equation*}
the collection of graphs with vertex- and edge-mark vectors given by $\vec u^{(n)}$ and $\vec{m}^{(n)}$, respectively.

\par For each $n \in \N$, let $G_n$ be uniformly sampled from $\mathcal{G}^{(n)}_{\vec m,\vec u}$, which we call the DA model in honor to Delgosha and Anantharam. In~\cite{da}, the authors introduce a notion of entropy for such random graphs, which they call BC entropy, and provide a weak large deviation principle for the empirical measure of $G_{n}$, denoted by $U(G_{n})$ (see~\eqref{eq:empirical_neighborhood_distribution}).

\par The asymptotic behavior of these distributions depends on the limits of the vectors $\vec{u}^{(n)}$ and $\vec{m}^{(n)}$. With this in mind, we introduce some additional notation for these limits. The average-degree vector is a vector of non-negative entries $\vec{d} = \big( d_{x,x'}: x, x' \in \Xi \big)$ such that 
\begin{equation*}
d_{x,x'}=d_{x',x}, \text{ for all } x,x' \in \Xi, \quad \text{and} \quad \sum_{x,x'}d_{x,x'}>0.
\end{equation*}
The vector $\vec{d}$ plays the role of the limit of $\vec{m}^{(n)}/n$. Likewise, we introduce a distribution $Q=(q_\theta)_{\theta\in \Theta}$ over $\Theta$ that will be regarded as the limit of $\vec{u}^{(n)}/n$.

\begin{definition}\label{def:adapted}
Given an average-degree vector $\vec d$ and a probability distribution $Q = (q_\theta)_{\theta \in \Theta}$, we say that a sequence $(\vec m^{(n)} ,\vec u^{(n)})$ of edge- and vertex-mark count vectors is \textit{adapted} to $( \vec d,Q)$, if the following conditions hold:
\begin{enumerate}
\item $\pnorm{1}{\vec m^{(n)}}\leq \binom{n}{2}$ and $\pnorm{1}{\vec u^{(n)}}=n$. This guarantees that $\mathcal{G}^{(n)}_{\vec m,\vec u}$ is not empty.
\item $m^{(n)}(x,x)/n\to d_{x,x}/2$, for all $x \in \Xi$.
\item $m^{(n)}(x,x')/n\to d_{x,x'}=d_{x',x}$, for all $x \neq x' \in \Xi$.
\item $u^{(n)}(\theta)/n\to q_\theta$, for all $\theta \in \Theta$.
\item If $d_{x,x'}=0$, then $m^{(n)}(x,x')=0$, for all $n$.
\item $q_\theta=0$ implies $u^{(n)}(\theta)=0$, for all $n$.
\end{enumerate}
\end{definition}

\noindent \textbf{The large deviation principle.}
Let $\mu \in \mathcal{P}(\bar{\mathcal{G}}_*)$ and denote by $B(\mu,\varepsilon)$ the ball of radius $\varepsilon$ in the L\'{e}vy-Prokhorov metric in $\mathcal{P}(\bar{\mathcal{G}}_*)$. The weak large deviation principle concerns the exponential behavior of the sequence of probabilities $\P(U(G_n)\in B(\mu,\varepsilon))$.

\par Consider the set
\begin{equation*}
\mathcal{G}^{(n)}_{\vec m,\vec u}(\mu,\varepsilon):= \left\{ G\in \mathcal{G}^{(n)}_{\vec m,\vec u}\,:\,\d_{LP}(U(G),\mu)<\varepsilon \right\}.
\end{equation*}
Since $G_n$ is uniformly distributed over $\mathcal{G}^{(n)}_{\vec{m},\vec{u}}$, it follows by definition that
\begin{equation*}
\log \P(U(G_n)\in B(\mu,\varepsilon))=\log \big|\mathcal{G}^{(n)}_{\vec m,\vec u}(\mu,\varepsilon)\big|-\log \big|\mathcal{G}^{(n)}_{\vec m,\vec u}\big|.
\end{equation*}

\par If $(\vec m^{(n)} ,\vec u^{(n)})$ is adapted to $(\vec d, Q)$, it is proved in \cite[Equation~10]{da} that 
\begin{equation}\label{eq:sizeGmu}
\log \big|\mathcal{G}^{(n)}_{\vec m,\vec u}\big|=\big\|{\vec m^{(n)}}\big\|_{1}\log n+nH(Q)+n\sum_{x,x'}s(d_{x,x'})+o(n),
\end{equation}
where $s(d)=d/2-(d/2)\log d$, if $d>0$, and $s(d)=0$ if $d=0$, and $H(Q) = -\sum_{\theta}Q_{\theta} \log Q_{\theta}$ denotes the entropy of the probability measure $Q$.

\par Delgosha and Anantharam~\cite{da} introduced the notion of $\varepsilon$-upper BC entropy
\begin{equation}\label{eq:def_entropy}
\overline \Sigma_{\vec d,Q}(\mu,\varepsilon)\vert_{\vec m^{(n)},\vec u^{(n)}}:=\limsup_{n\to\infty}\frac{\log |\mathcal{G}^{(n)}_{\vec m,\vec u}(\mu,\varepsilon)|-\pnorm{1}{\vec m^{(n)}}\log n}{n}.
\end{equation}
Taking the limit $\varepsilon\to 0$, the upper BC entropy can be defined as
\begin{equation*}
\overline \Sigma_{\vec d,Q}(\mu)\vert_{\vec m^{(n)},\vec u^{(n)}}=\lim_{\varepsilon \downarrow 0}\overline \Sigma_{\vec d,Q}(\mu,\varepsilon)\vert_{\vec m^{(n)},\vec u^{(n)}}.
\end{equation*}
Similar definitions hold when we change the superior limits to inferior limits and, in this case, we change the notation, substituting $\overline{\Sigma}$ by $\underline{\Sigma}$.

For an average degree vector $\vec{d}$, we denote by $s(\vec{d}\,)$ the quantity
\begin{equation}\label{def:s}
s(\vec{d}\,)=\sum_{x,x'}s(d_{x,x'}).
\end{equation}

The next result establishes the central piece of the weak large deviation principle. It proves that the upper and lower BC entropies coincide and that they do not depend on the adapted sequence $(\vec{m}^{(n)}, \vec{u}^{(n)})$.
\begin{theorem}[\cite{da}, Theorem~2]\label{t:equality_entropy}
Let $\mu\in \mathcal{P}(\bar{\mathcal{G}}_*)$ with $0<\deg(\mu)<\infty$. The following statements hold.
\begin{enumerate}
\item The values of $\overline \Sigma_{\vec d,Q}(\mu)\vert_{\vec m^{(n)},\vec u^{(n)}}$ and $\underline \Sigma_{\vec d,Q}(\mu)\vert_{\vec m^{(n)},\vec u^{(n)}}$ do not depend on the adapted sequence $(\vec{m}^{(n)},\vec{u}^{(n)})$. For this reason, we simplify the notation by writing $\overline \Sigma_{\vec d,Q}(\mu)$ and $\underline \Sigma_{\vec d,Q}(\mu)$. 
\item The equality $\overline \Sigma_{\vec d,Q}(\mu)=\underline \Sigma_{\vec d,Q}(\mu)$ holds. This value is simply written as $ \Sigma_{\vec d,Q}(\mu)$ and it belongs to $[-\infty,s(\vec d\,)+H(Q)]$.
\end{enumerate}
\end{theorem}

The quantity $\Sigma_{\vec d,Q}(\mu)$ is called the BC entropy of $\mu$ associated with the pair $(\vec d,Q)$. From Theorem~\ref{theorem1}, we conclude that unless $\vec{d} = \vec{\deg}(\mu)$, $Q =\vec{\Pi}(\mu)$, and $\mu$ is a unimodular measure on $\bar{\mathcal{T}}_*$, we have $\Sigma_{\vec{d},Q} (\mu) = -\infty$. Additionally, the function $\mu \mapsto \Sigma_{\vec d,Q}(\mu)$ is upper semicontinuous on the set $\mathcal{P}_u(\bar{\mathcal{G}}_*)$ \cite[Lemma~5]{da}.

Combining the theorem above, Equation~\eqref{eq:sizeGmu} and~\cite[Theorem 4.1.11]{dembo_zeitouni}, one readly concludes a weak large deviation principle for the collection $\big(U(G_{n})\big)_{n \in \N}$, with rate function given by
\begin{equation}\label{eq:I_d_Q}
I_{\vec{d},Q}(\mu)= H(Q)+s(\vec{d}\,)-\Sigma_{\vec d,Q} (\mu).
\end{equation}

We now verify that this family actually satisfies a strong large deviation principle.
\begin{theorem}\label{t:ldp}
Assume that the spaces of marks $\Theta$ and $\Xi$ in the definition of $\mathcal{G}^{(n)}_{\vec{m},\vec{u}}$ are finite, and let $G_n$ be sampled uniformly from $\mathcal{G}^{(n)}_{\vec{m},\vec{u}}$ with $(\vec{m}^{(n)},\vec{u}^{(n)})$ adapted to $(\vec d,Q)$. The sequence $\big(U(G_{n})\big)_{n \in \N}$ satisfies a large deviation principle with rate function given by~\eqref{eq:I_d_Q}.
\end{theorem}

\begin{remark}
Conditions 5 and 6 in Definition~\ref{def:adapted} of adapted sequence are unnecessary. We claim that if these conditions are not verified then it is possible to define an exponentially equivalent model for which such conditions hold. Indeed, if either $m^{(n)}(x,x')/n\to 0$ or $u^{(n)}(\theta)/n\to 0$ then we need to modify only $o(n)$ marks on vertices and edges in order to comply with Conditions 5 and 6 in Definition~\ref{def:adapted}. Due to our sparsity assumption these changes only affect $o(n)$ neighborhoods. This argument can be made precise following the steps of the proof of Lemma~\ref{lemma:distace_estimate} below.
\end{remark}

\begin{proof}[Proof of Theorem~\ref{t:ldp}]
It suffices to verify exponential tightness for the family $\big(U(G_{n})\big)_{n \in \N}$. From Lemma~\ref{le:compacts} in Appendix~\ref{app:compacity}, we can construct for any compact set $\Pi \subset\mathcal{G}_{*}$, a corresponding  compact set $\bar{\Pi} \subset \bar{\mathcal{G}}_{*}$ such that $U(\bar{G}) \in \bar{\Pi}$ if, and only if, $U(\pi_g(\bar{G})) \in \Pi$. In particular, exponential tightness of the sequence $U(\bar{G}_n)$ follows directly from the exponential tightness of the sequence $U(\pi_g(\bar{G}_n))$. This is a consequence of~\cite[Lemma~6.2]{bc} since, if $\bar{G}_n$ is uniformly distributed from $\mathcal{G}_{\vec{m},\vec{u}}^{(n)}$, then $G_n=\pi_g(\bar{G}_n)$ is uniformly distributed from $\mathcal{G}_{n,m}$, the set of graphs with $n$ vertices and $m=\|\vec{m}\|_{1}$ edges.
\end{proof}

In Section~\ref{sec:discrete_marks} we combine Theorem~\ref{t:ldp} with mixture techniques from Appendix~\ref{appendix:mixture_and_LDP}. For this reason, it is convenient to parametrize the model not via the pair $(\vec{d},Q)$ like in the DA model, but actually replace $\vec{d}$ with a probability measure. In the following, we show how to recover Theorem~\ref{t:ldp} with the new parameters.

 Motivated by this we fix a parameter $d>0$, $P\in\mathcal{P}(\Xi^2_\leq)$ and $Q\in\mathcal{P}(\Theta)$. We define $\vec{d}=\vec{d}(d,P)$ via: for $x\leq x'$
 \begin{equation}\label{def:Pd}
d_{x,x'}= \begin{cases}
 dP(x,x) &, \text{ if }x=x'\\
 \frac{d}{2}P(x,x') &, \text{ otherwise,}
 \end{cases}
\end{equation}
and $d_{x,x'}=d_{x',x}$ if $x>x'$.

With the above definition, $(\vec{m}^{(n)},\vec{u}^{(n)})$ is adapted (except for conditions 5 and 6) to $(\vec{d},Q)$ if and only if $\|\vec{m}_n\|/n\to d/2$ and $(\vec{m}_\leq^{(n)}/\|\vec{m}_n\|,\vec{u}^{(n)}_n/n)$ converges weakly to $(P,Q)$:
\begin{equation}
\begin{split}
P(x,x)&=\lim_n\frac{m_n(x,x)}{m_n}=\lim_n\frac{m(x,x)}{n}\frac{n}{m_n}=\frac{d_{x,x}}{2}\frac{2}{d}=\frac{d_{x,x}}{d}\\
P(x,x')&=\lim_n\frac{m_n(x,x')}{m_n}=\lim_n\frac{m(x,x')}{n}\frac{n}{m_n}=d_{x,x'}\frac{2}{d}=\frac{2d_{x,x'}}{d}
\end{split}
\end{equation}
This reasoning also leads to the relation
\begin{equation}
d=2\sum_{x<x'}d_{x,x'}+\sum_xd_{x,x}\,.
\end{equation}

From Theorem~\ref{t:ldp} we obtain the following corollary.
\begin{cor}\label{c:ldp}
  Assume that the spaces of marks $\Theta$ and $\Xi$ are finite.
  Fix $d>0$, $P\in\mathcal{P}(\Xi^2_\leq)$, and $Q\in\mathcal{P}(\Theta)$.
  Consider sequences $(\vec{m}_n,\vec{u}_n)$ such that $\|\vec{m}_n\|_1/n \to d/2$ and $(\vec{m}_\leq^{(n)}/\|\vec{m}_n\|_1,\vec{u}^{(n)}_n/n)$ converges weakly to $(P,Q)$.
  If $G_n$ is sampled uniformly from $\mathcal{G}^{(n)}_{\vec{m},\vec{u}}$, then the sequence $\big(U(G_{n})\big)_{n \in \N}$ satisfies a large deviation principle with rate function given by
\begin{equation}
I_{P,Q}(\mu)= H(Q)+s(\vec{d}(P,d)\,)-\Sigma_{\vec d(P,d),Q} (\mu),
\end{equation}
where $\vec{d}=\vec{d}(d,P)$ is defined through~\eqref{def:Pd}.
\end{cor}

\subsection{Partitions and projections}
\label{sec:partproject}
  
Given tagged partitions\footnote{A tagged partition of a set $X$ is a collection $\mathcal{A} = \{ (A_{1}, a_{1}), \dots, (A_{\ell}, a_{\ell}) \}$ such that $\cup_{i=1}^{\ell}A_{i} = X$ and $a_{i} \in A_{i}$, for all $ i \leq \ell$.} $\mathcal{A} = \{ (A_{1}, a_{1}), \dots, (A_{\ell}, a_{\ell}) \}$ of $\Theta$ and $\mathcal{B} = \{ (B_{1}, b_{1}), \dots, (B_{\ell'}, b_{\ell'}) \}$ of $\Xi$, consider the projection maps
\begin{equation}
\pi_{\mathcal{A}}: \Theta \to \{a_{1}, \dots, a_{\ell} \} \text{ and } \pi_{\mathcal{B}}: \Xi \to \{b_{1}, \dots, b_{\ell'} \}.
\end{equation}
  For a graph $\bar{G}=(G,\vec{\tau},\vec{\xi} \,)$ with marks in the spaces $\Theta$ and $\Xi$, we define its projection $\bar{G}^{\pi}=(G,\pi_{\mathcal{A}}(\vec{\tau}),\pi_{\mathcal{B}}(\vec{\xi}) \,)$ where
\begin{equation}
\pi_{\mathcal{A}}(\vec{\tau})=\big(\pi_{\mathcal{A}}\circ\tau(v) \big)_{v\in V} \quad \text{and} \quad \pi_{\mathcal{B}}(\vec{\xi}) = \big(\pi_{\mathcal{B}} \circ \xi(\vec{e} \, ) \big)_{\vec{e} \in \vec{E}}.
\end{equation}
In words, $\bar{G}^{\pi}$ is the marked graph which has the same underlying graph as $\bar{G}$  and whose marks are obtained by projecting the marks of $\bar{G}$ using the maps $\pi_{\mathcal{A}}$ and $\pi_{\mathcal{B}}$. This projection map extends naturally for the set of rooted marked graphs up to isomorphisms $\bar{\mathcal{G}}_*(\Theta, \Xi) \to \bar{\mathcal{G}}_* \big( \pi_{\mathcal{A}}(\Theta), \pi_{\mathcal{B}}(\Xi) \big)$.

For a probability measure $\mu$ on $\Theta$, we denote by $\mu^{\pi_{\mathcal{A}}}$ its pushforward to the space $\{a_{1}, \dots, a_{\ell} \}$. Later on, we will work with probabilities $\chi$ defined in $\Xi \times \Xi$. In this case, the measure $\chi^{\pi_{\mathcal{B}}}$ denotes the pushforward of $\chi$ into the space $\{b_{1}, \dots, b_{\ell'} \} \times \{b_{1}, \dots, b_{\ell'} \}$ via the coordinate map $\pi_{\mathcal{B}} \times \pi_{\mathcal{B}}$.

Fix two measures $\nu\in \mathcal{P}(\Theta)$ and $\chi\in\mathcal{P}(\Xi^2)$ and two positive real numbers $\epsilon,\delta>0$. We say that the pair of partitions $(\mathcal{A},\mathcal{B})$ is an $(\epsilon,\delta)$-good pair with respect to $(\nu,\chi)$ if the following two condition holds:
\begin{enumerate}
\item The diameter of the sets are smaller that $\epsilon$ with exception for the last one:
\begin{equation}
\begin{split}
\diam_\Theta(A_i) & \leq \epsilon, \quad \text{for }  i \leq\ell-1, \text{ and }\\
\diam_{\Xi^2}(B_i \times B_{j}) & \leq \epsilon, \quad \text{for } i, j \leq \ell'-1.
\end{split}
\end{equation}
\item The sets that are not small with respect to diameter are small in measure:
\begin{equation}
\nu( A_\ell ) < \delta \quad \text{and} \quad \chi \Big( ( B_{\ell'} \times \Xi ) \cup ( \Xi \times B_{\ell'} ) \Big) < \delta.
\end{equation}
\end{enumerate}

The existence of good pairs of partitions for each $\epsilon,\delta>0$ is a consequence of Prokhorov's Theorem and we leave to the reader to check the details of this construction.

\section{Proof preparation}\label{sec:prep}

\par This section contains the statement of the main propositions that are necessary in order to prove Theorems~\ref{t:ldp_uniformg_general} and~\ref{t:ldp_erg_general} as well as some additional notation and precise construction of the models.

\subsection{The models}\label{sec:model}

\par We fix $d>0$ and a sequence $m_n\in \mathbb{N}$ such that
\begin{equation}
\lim_{n} \frac{m_n}{n} = \frac{d}{2}\,.
\end{equation}
Recall that $\mathcal{G}_{n,m_n}$ is the set of graphs with $n$ vertices and $m_{n}$ edges.

  Let $G_{n}$ be a random graph ensemble, which can be either uniformly sampled over $\mathcal{G}_{n,m_n}$ or an Erd\H{o}s-R\'enyi random graph $G(n, p)$, with $p = \frac{d}{n}$.
  Fix two distributions $\nu \in \mathcal{P}(\Theta)$ and $\chi\in\mathcal{P}(\Xi\times\Xi)$.
  Given the underlying random graph $G_{n}$, we define the marks on $G_{n}$ with distributions $\nu$ and $\chi$ in the following way:
\begin{enumerate}
\item To each vertex $v\in G_n$ we associate a mark $O_v\in \Theta$ and we assume that $(O_v)_{v\in G_n}$ are i.i.d. with law $\nu$.
\item To each edge $e=\{u,v\} \in G_n$ we first sample, independently from the previous step, a mark $X_{e}=(X^1_e,X^2_e)\in \Xi\times\Xi$ and we assume that $(X_e)_e$ are i.i.d. with law $\chi$.
  After that, with probability $1/2$ the mark of the oriented edge $(u,v)$ is $X^1_e$ and the mark of the edge $(v,u)$ is $X^2_e$.
  With  probability $1/2$ the mark of $(u,v)$ is $X^2_e$ and the mark of $(v,u)$ is $X^1_e$.  
\end{enumerate}

\subsection{The case of discrete marks}\label{sec:results}

\par In this subsection we state the analog of Theorem~\ref{t:ldp_erg_general} for the case when $\Theta$ and $\Xi$ are finite sets.

For $\mu \in \mathcal{P}(\bar{\mathcal{G}}_*)$, define $P(\vec{\deg}(\mu))$ via~\eqref{def:Pd} and
\begin{equation}\label{eq:rate_function_uniform_graph}
\lambda_{d}(\mu)= I_{P(\vec{\deg}(\mu)), \Pi(\mu)}(\mu) + \frac{d}{2}H \big( P(\vec{\deg}(\mu)) \big| \chi_\leq \big) + H(\Pi(\mu)|\nu)\,,
\end{equation}
where $\chi_\leq(x,x')=(\chi(x,x')+\chi(x',x))/2$ for $x\leq x'$.

\begin{theorem}\label{t:ldp_uniformg}
  Assume that $\Theta$ and $\Xi$ are finite sets and let $\nu \in \mathcal{P}(\Theta)$ and $\chi\in\mathcal{P}(\Xi\times\Xi)$.
  Consider $G_n$ the marked graph with i.i.d.\ discrete marks (distributed according to $\nu$ and $\chi$ for vertices and edges, respectively) and underlying graph uniformly sampled from $\mathcal{G}_{n,m_n}$.
  The sequence of empirical neighborhood measures $U(G_n)$ satisfies a large deviation principle with a good rate function $I^u_{d}:\mathcal{P}(\bar{ \mathcal{G}}_*)\to [0,\infty]$ given by
\begin{equation}\label{eq:rate_function_uniform_graph_2}
I^u_{d}(\mu)= \begin{cases}
\lambda_{d}(\mu), \text{ if } \deg(\mu) = d, \\
+\infty, \text{ otherwise}.
\end{cases}
\end{equation}
\end{theorem}

The next theorem provides the analogous result if the underlying graph ensemble is given by a sparse Erd\H{o}s-R\'enyi random graph.
\begin{theorem}\label{t:ldp_erg}
  Assume that $\Theta$ and $\Xi$ are finite sets and let  $\nu \in \mathcal{P}(\Theta)$ and $\chi\in\mathcal{P}(\Xi\times\Xi)$.
  Let $G_n$ be the marked graph with i.i.d.\ discrete marks (with distributions $\nu$ and $\chi)$ and underlying graph the Erd\H{o}s-R\'enyi random graph $G(n,p)$ with $p=\tfrac{d}{n}$.
  The sequence of empirical neighborhood measures $U(G_n)$ satisfies a large deviation principle with a good rate function $I_{d}^{\rm ER}:\mathcal{P}(\bar G_*)\to [0,\infty]$ given by
\begin{equation}\label{eq:rate_function_ER_2}
I_{d}^{\rm ER}(\mu)=I^u_{\deg(\mu)}(\mu)+\frac{1}{2}\Big( \deg(\mu) \log \frac{\deg(\mu)}{d}- \deg(\mu) +d \Big).
\end{equation}
\end{theorem}

The proofs of these results build on top of the work of Delgosha and Anantharam~\cite{da}, in particular, Corollary~\ref{c:ldp}. Theorems~\ref{t:ldp_uniformg} and~\ref{t:ldp_erg} follow by arguments regarding mixtures of measures (see Theorem~\ref{t:mixture} in the Appendix). 
In order to verify Theorem~\ref{t:ldp_uniformg}, we prove that the model of uniform marked graph with given number of edges and i.i.d.\ marks can be written as a mixture of the DA model, when the parameters $P$ (see ~\eqref{def:Pd}) and $Q$ are chosen at random with appropriate distributions. When this is done, Theorem~\ref{t:ldp_uniformg} follows from a general statement regarding large deviations for probabilities that can be written as mixtures. Theorem~\ref{t:ldp_erg} also follows the same lines, by noticing that the Erd\H{o}s-R\'eniy random graph can be obtained by choosing a random number of edges $m$ (distributed as a Binomial$\big( \binom{n}{2}, p \big)$ random variable) and then independently sampling a uniform graph with $m$ edges. Theorem~\ref{t:ldp_uniformg} and Theorem~\ref{t:mixture} can now be combined to conclude Theorem~\ref{t:ldp_erg}.

\subsection{Rate functions for the general case}\label{sec:mainresults2}

\par We are now in position to provide expressions for the rate functions in Theorems~\ref{t:ldp_uniformg_general} and~\ref{t:ldp_erg_general}. Recall the notion of $(\epsilon, \delta)$-good partitions in Section~\ref{sec:partproject}. For each $k$, we consider a $\big( \tfrac{1}{2k}, e^{-k^{2}} \big)$-good tagged partition of $\Theta$ and $\Xi$, which we call $\mathcal{A}_{k}$ and $\mathcal{B}_{k}$.
  Recall the projected graph introduced in Section~\ref{sec:partproject} and denote, for each fixed $k$, by $G_{n}^{k}$ the graph obtained by projecting the marks on vertices with the map $\pi_{\mathcal{A}_{k}}$ and marks on edges with the map $\pi_{\mathcal{B}_{k}}$.
  Define now, for $\mu \in \mathcal{P}\big( \bar{\mathcal{G}}_* ( \Theta, \Xi ) \big)$,
\begin{equation}\label{eq:rate_function_uniform_graph_discretized}
\lambda^{k}_{d}(\mu)= I^{k}_{P(\vec{\deg}(\mu)), \Pi(\mu)}(\mu) + \frac{d}{2}H \big( P(\vec{\deg}(\mu)) \big| \chi_\leq^{\pi_{\mathcal{B}_k}} \big) + H(\Pi(\mu)|\nu^{\pi_{\mathcal{A}_k}}),
\end{equation}
where $I^{k}_{P(\vec{\deg}(\mu)), \Pi(\mu)}(\mu)$ denotes rate function given by Theorem~\ref{t:ldp} when the spaces of marks are $\pi_{\mathcal{A}_{k}}(\Theta)$ and $\pi_{\mathcal{B}_{k}}(\Xi)$. Notice that the quantity above is infinite unless $\mu$ is supported on $\bar{\mathcal{G}}_* \big(\pi_{\mathcal{A}_{k}}(\Theta), \pi_{\mathcal{B}_{k}}(\Xi) \big)$.

  Theorem~\ref{t:ldp_uniformg} states now that the sequence $\big( U(G^{k}_{n}) \big)_{n\in \mathbb{N}}$ satisfies a large deviation principle on the space $\mathcal{P}\big( \bar{\mathcal{G}}_* (\Theta, \Xi) \big)$ with rate function defined by, for $\mu \in \mathcal{P}\big( \bar{\mathcal{G}}_* (\Theta, \Xi) \big)$,
\begin{equation}
I_{d}^{u,k}(\mu)= \begin{cases}
\lambda_{d}^{k}(\mu), \text{ if } \deg(\mu) = d, \\
+\infty, \text{ otherwise}.
\end{cases}
\end{equation}

The rate function in Theorem~\ref{t:ldp_uniformg_general} is now given by
\begin{equation}\label{eq:liminfliminf}
\bar{I}_d^{u}(\mu) = \limsup_{\delta \to 0} \limsup_{k \to \infty} \inf_{ \nu \in B_{\delta}(\mu) } I_d^{u,k}(\nu).
\end{equation}

For $\bar{I}_d^{\rm ER}$, consider as above $\mathcal{A}_{k}$ and $\mathcal{B}_{k}$ $\big( \tfrac{1}{2k}, e^{-k^{2}} \big)$-good tagged partitions of $\Theta$ and $\Xi$ adapted to $\nu \in \mathcal{P}(\Theta)$ and $\chi\in \mathcal{P}(\Xi^{2})$, for each integer $k \geq 1$.
  Let $G_{n}$ denote the sparse Erd\H{o}s-R\'enyi random graph $G(n,p)$, with $p=\tfrac{d}{n}$, and i.i.d.\ marks on vertices and edges with distribution $\nu$ and $\chi$, respectively. Theorem~\ref{t:ldp_erg} states that the sequence of graphs with projected marks $\big( U(G^{k}_{n}) \big)_{n\in \mathbb{N}}$ satisfies a large deviation principle on the space $\mathcal{P}\big( \bar{\mathcal{G}}_* (\Theta, \Xi) \big)$ with rate function given by, for $\mu \in \mathcal{P}\big( \bar{\mathcal{G}}_* (\Theta, \Xi) \big)$,
\begin{equation}
I_{d}^{{\rm ER},k}(\mu)=I^{u,k}_{\deg(\mu)}(\mu)+\deg(\mu) \log \frac{2\deg(\mu)}{d}- \deg(\mu) +\frac{d}{2}. 
\end{equation}

Analogously to the setting of Theorem~\ref{t:ldp_uniformg_general}, we define
\begin{equation}\label{eq:liminfliminf_ER}
\bar{I}_d^{\rm ER}(\mu) = \limsup_{\delta \to 0} \limsup_{k \to \infty} \inf_{ \nu \in B_{\delta}(\mu) } I_d^{{\rm ER },k}(\nu).
\end{equation}

\begin{remark} Observe also the following relation
\begin{equation*}
\bar{I}_d^{\rm ER}(\mu)=\bar{I}_d^{u}(\mu)+\deg(\mu) \log \frac{2\deg(\mu)}{d}- \deg(\mu) +\frac{d}{2}.
\end{equation*}
\end{remark}

\subsection{Minimizers}

We now turn to the problem of identifying minimizers of the rate functions, starting with the DA model, where a minimizer can be obtained as a marked unimodular Galton-Watson tree with seed $\PP^{*}_{d}$ that we introduce now.

In the context of graphs without marks,~\cite[Corollary 1.4]{bc} implies that the minimizer is given by a unimodular Galton-Watson random tree with offspring distribution $\poisson(d)$. A natural guess is then to expect that the same holds for the underlying graph in the DA model. Indeed, let $\PP^{*}_{d} \in \mathcal{P}_1$ denote the distribution of marked random graph where the number of children of the root is distributed according to $\poisson(d)$. We then independently attribute marks to vertices and edges with distribution $Q$ and $d_{x,x'}/d$ (think of attributing marks to oriented edges), respectively.

Equivalently, we define, for $t \in \bar{\mathcal{T}}_{*}^{1}$,
\begin{equation*}
N_{x,x'}^{\theta,\theta'}(t,o)=|\{v\sim_t o\,:\,\tau_t(o) = \theta, \tau_t(v) = \theta', \xi_t(o,v)=x', \xi_t(v,o)=x\}|,
\end{equation*}
the distribution $\PP^{*}_{d}$ is given by
\begin{equation}\label{def:P*}
\PP^*_{d}(t)=Q(\theta_t)\prod_{\theta', x, x'}{\rm Poi}(\alpha_{x,x'}^{\theta'})(N_{x,x'}^{\theta_t,\theta'}(t)),
\end{equation}
where $\theta_t$ is the mark on the root of $t$, $\alpha_{x,x'}^{\theta'}=Q(\theta')d_{x,x'}$, and for $\lambda>0$ and $n \in \N$, ${\rm Poi}(\lambda)(n) = e^{-\lambda}\frac{\lambda^{n}}{n!}$ is the probability that a random variable with distribution Poisson with parameter $\lambda$ equals $n$.

The proposition below states that the seed above indeed provides a minimizer for the DA model.
\begin{proposition}\label{prop:minimizer}
The minimizer of the rate function $I_{\vec{d}, Q}$ is the marked unimodular Galton-Watson tree with seed $\PP^{*}_{d}$.
\end{proposition}

The proof of this proposition follows from explicit calculations relying on the expression for the distribution $\PP^*_{d}$. In view of Proposition~\ref{prop:entropyUGW}, it suffices to prove that $J_1(\PP^*_{d})=H(Q)+s(\vec d)$, with $J_{1}$ as in~\eqref{eq:J_h}.

With this proposition in hands, we can easily combine it with the expressions for the rate functions in the contexts presented above to deduce expressions for the rate functions.

Assume first that the spaces $\Theta$ and $\Xi$ are finite. Given two distributions $\nu$ and $\chi$ in $\Theta$ and $\Xi \times \Xi$, respectively, let $\PP^{*}_{\nu, \chi, d} \in \mathcal{P}_1$ denote the distribution of marked random tree of depth one where the root has degree distributed according to ${\rm Poi}(d)$ and independent marks on vertices and edges with respective distributions $\nu$ and $\chi$. In the collection of edge marks, if both marks are distinct, we independently sample an orientation for the edge uniformly.
\begin{proposition}\label{prop:minimizer_discrete}
Assume that the spaces $\Theta$ and $\Xi$ are finite. The marked unimodular Galton-Watson tree with seed $\PP^{*}_{\nu, \chi, d}$ is a minimizer of the rate functions $I^{u}_{d}$ and $I^{\rm ER}_{d}$.
\end{proposition}

In the case when the spaces of marks are not discrete, one can still define the seeds $\PP^{*}_{\nu, \chi, d}$ and the next proposition states that the marked unimodular Galton-Watson tree with seed $\PP^{*}_{\nu, \chi, d}$ is still the minimizer of the rate functions.
\begin{proposition}\label{prop:minimizer_general}
The marked unimodular Galton-Watson tree with seed $\PP^{*}_{\nu, \chi, d}$ is a minimizer of the rate functions $\bar{I}^{u}_{d}$ and $\bar{I}^{\rm ER}_{d}$.
\end{proposition}

The proof of Proposition~\ref{prop:minimizer_discrete} will follow from a combination of Proposition~\ref{prop:minimizer} together with the expressions for $I^{u}_{d}$ and $I^{\rm ER}_{d}$ in~\eqref{eq:rate_function_uniform_graph_2} and~\eqref{eq:rate_function_ER_2}.
  
The proof of Proposition~\ref{prop:minimizer_general}, the case where the marks lie in general Polish spaces, requires a step further. Indeed, in this case the rate function is obtained as an approximation of the rate function of the discrete case. We verify that the minimizers of the discretized rate functions converge to the candidate minimizer, the marked unimodular Galton-Watson tree with seed $\PP^{*}_{\nu, \chi, d}$. This will be done by constructing a coupling between the two measures which is obtained by simply projecting the marks according to the projection maps introduced together with the good partitions, in the same spirit as in Lemma~\ref{lemma:distace_estimate}.

\section{From discrete to general spaces}\label{sec:discrete_to_general}

\par This section contains the proofs of Theorems~\ref{t:ldp_uniformg_general} and~\ref{t:ldp_erg_general} as well as the proof of Lemma~\ref{lemma:distace_estimate}.

\subsection{Proof of Theorems~\ref{t:ldp_uniformg_general} and~\ref{t:ldp_erg_general}}

\par The proofs of both theorems follow from the general lemma presented below which applies for sequence of sparse graphs in the sense of Definition~\ref{def:sparsity}. This lemma provides necessary conditions in order to extend the result of large deviations with discrete marks to graphs with marks in general Polish spaces.

\begin{definition}\label{def:sparsity} Let $(G_n)_{n\in \mathbb{N}}$ be a sequence of random graphs. We say that $(G_n)_{n\in \mathbb{N}}$ is \textit{uniformly exponential tight} if, for each $\varepsilon>0$,
\begin{equation}\label{eq:many_limits}
\limsup_{h \to \infty} \limsup_{S \to \infty} \limsup_{n\to\infty} \frac{1}{n}\log \P\bigg(\frac{1}{n}\sum_{v \in V_{n}}\mathbf{1}_{\{|[G_n,v]_{h}|>S\}} > \varepsilon \bigg) = -\infty
\end{equation}
and
\begin{equation}\label{eq:bound_edges}
\limsup_{a \to \infty} \limsup_{n \to \infty} \frac{1}{n} \log \P\big( |E_{n}| \geq an \big) = -\infty.
\end{equation}
\end{definition}
\begin{lemma}\label{lemma:black_box}
Let $G_{n}$ be a random marked graph with i.i.d.\ marks on edges and vertices distributed according to $\nu \in \mathcal{P}(\Theta)$ and $\chi \in \mathcal{P}(\Xi^{2})$. Assume that the underlying graph of $G_{n}$ is uniformly exponential tight (cf.~Definition~\ref{def:sparsity}). 
Consider, for each integer $k > 0$, $\big( \tfrac{1}{2k}, e^{-k^{2}} \big)$-good tagged partitions $(\mathcal{A}_{k}, \mathcal{B}_{k})$ of $\Theta$ and $\Xi$ with respect to $\nu$ and $\chi$, respectively. Denote by $G_{n}^{k}$ the projected marked graph and assume that, for each $k$, $U(G_{n}^{k})$ satisfies a large deviation principle with rate function $I^{k}$.

Then, $U(G_{n})$ satisfies a large deviation principle with rate function given by
\begin{equation}
I(\mu) = \limsup_{\delta \to 0} \limsup_{k \to \infty} \inf_{ \nu \in B_{\delta}(\mu) } I^{k}(\nu).
\end{equation}
\end{lemma}

Notice that  Theorems~\ref{t:ldp_uniformg_general} and~\ref{t:ldp_erg_general} follow immediately from the lemma above.
  Lemma~\ref{lemma:caputoadaptado} can be used to verify~\eqref{eq:many_limits} for the two models of that theorems, while~\eqref{eq:bound_edges} follows from Lemma~\ref{lemma:number_of_edges_ER} for the Erd\H{o}s-R\'enyi random graph. Assumption~\eqref{eq:bound_edges} is void for the uniform graph with $m_{n}$ edges, since $\tfrac{m_{n}}{n} \to \tfrac{d}{2}$ implies that the event that the graph has more than $an$ edges is empty if $a$ is large enough.

\begin{proof}[Proof of Lemma~\ref{lemma:black_box}]
  We begin by observing that the statement follow directly from Lemma~\ref{lema:Kurtz} if one verifies~\eqref{eq:hipothesis_kurtz}.
  In view of Lemma~\ref{lemma:distace_estimate}, for every $S \geq 1$ and $h \geq 8k$, the quantity of interest is bounded by the maximum of the following three expressions
\begin{equation}\label{eq:estimates}
\begin{split}
\limsup_{n\to\infty}\frac{1}{n}\log \P\bigg(\frac{1}{n}\sum_{v \in V_{n}}\mathbf{1}_{\{|[G_n,v]_{h}|>S\}} > \frac{1}{16k}\bigg) \\
\limsup_{n\to\infty}\frac{1}{n}\log\P\bigg(\frac{1}{n}\sum_{v \in V_{n}}\mathbf{1}_{\{ O_{v} \in A^{*} \}} > \frac{1}{16Sk} \bigg) \\
\limsup_{n\to\infty}\frac{1}{n}\log\P\bigg(\frac{1}{n}\sum_{e\in E_n}\mathbf{1}_{\{ X_{e} \in B^{*} \}} > \frac{1}{16Sk} \bigg),
\end{split}
\end{equation}
uniformly in $h \geq 8k$ and $S \geq 1$.

  To control the first superior limit, notice that, due to~\eqref{eq:many_limits}, this quantity converges to $-\infty$ if we take the limit as $S$ and $h$ grow.

  Let us now focus on the second expression above, which only depends on the number of vertices of the graph and not on the graph structure. In this case, we apply Lemma~\ref{lemma:leiBinomial} to conclude that
\begin{equation}
\begin{split}
\limsup_{k \to \infty} \limsup_{n\to\infty} & \frac{1}{n} \log \P \bigg(\frac{1}{n}\sum_{v\in V_{n}}\mathbf{1}_{\{O_{v} \in A^{*}\}} > \frac{1}{16Sk} \bigg) \\
& \leq \limsup_{k \to \infty} \limsup_{n \to \infty} \frac{1}{n} \log \big( 16Ske^{-k^{2}+1} \big)^{\lceil \frac{1}{16Sk}n \rceil} \\
& \leq \limsup_{k \to \infty} \frac{1}{16Sk} \log \big( 16Ske^{-k^{2}+1} \big) = -\infty.
\end{split}
\end{equation}

  The last superior limit in~\eqref{eq:estimates} will also follow from Lemma~\ref{lemma:leiBinomial}.
  Indeed, notice that
\begin{equation}
\begin{split}
\P\bigg(\frac{1}{n}\sum_{e\in E_n}\mathbf{1}_{\{ X_{e} \in B^{*} \}}> \frac{1}{16Sk} \bigg) & \leq \P\bigg(\frac{1}{n}\sum_{e\in E_n}\mathbf{1}_{\{ X_{e} \in B^{*} \}}> \frac{1}{16Sk}, |E_{n}| \leq an \bigg) + \P\big( |E_{n}| \geq an \big) \\
& \leq \big( 16Sake^{-k^{2}+1} \big)^{\lceil \frac{1}{16aSk}n \rceil} + \P\big( |E_{n}| \geq an \big),
\end{split}
\end{equation}
which gives the bound
\begin{equation}
\limsup_{k \to \infty} \limsup_{n\to\infty} \frac{1}{n} \log \P\bigg(\frac{1}{n}\sum_{e\in E_n}\mathbf{1}_{\{ X_{e} \in B^{*} \}} > \frac{1}{16Sk} \bigg) \leq \limsup_{n \to \infty} \frac{1}{n} \log \P\big( |E_{n}| \geq an \big).
\end{equation}
  Since the estimate above is uniform in $a$, our result follows directly by applying Lemma~\ref{lemma:number_of_edges_ER} and taking the limit as $a$ grows.
  This concludes the proof.
\end{proof}

\subsection{Proof of the approximation result}\label{subsec:proof_main_lemma}

\par In this section we prove Lemma~\ref{lemma:distace_estimate}. For simplicity, we write $V$ and $E$ instead of $V(G)$ and $E(G)$, respectively.  Consider two $(\epsilon,\delta)$-good tagged partitions $\mathcal{A} = \{ (A_{1}, a_{1}), \dots, (A_{\ell}, a_{\ell}) \}$ of $\Theta$ and $\mathcal{B} = \{ (B_{1}, b_{1}), \dots, (B_{\ell'}, b_{\ell'}) \}$ of $\Xi$ (see Section~\ref{sec:partproject}). Recall the sets
\begin{equation}
A^*=A_\ell\,\text{ and } B^*=( B_{\ell'} \times \Xi ) \cup ( \Xi \times B_{\ell'} ).
\end{equation}

Fix $f:\bar{\mathcal{G}}_*\to\mathbb{R}$ with $\pnorm{BL}{f} \leq 1$. For $h \in \N$ fixed we derive, using the triangular inequality,
\begin{equation}\label{eq:BL_norm}
|U(G)_h(f)-U(G^{\pi})_h(f)| \leq \frac{1}{n}\sum_{v \in V}|f([G,v]_h)-f([G^{\pi},v]_h)|.
\end{equation}

We divide the proof in two steps. On the first step we define the set of vertices $v$ such that all the marks of $[G,v]_h$ fall in sets of small diameter. In the second step we control the influence of a mark that does not lie in a set of small diameter.

\textit{First step.} Define the subset of vertices
\begin{equation*}
\Gamma=\left\{\begin{array}{c} v\in V\,:\,\tau_w \in \Theta\setminus A^{*}, \text{ for all } w \in V([G,v]_h), \\ \text{and } \xi_e\in \Xi^2 \setminus B^{*}, \text{ for all } e \in E([G,v]_{h})
\end{array}\right\}.
\end{equation*}
We now prove the following bound
\begin{equation}\label{eq:disc-cont_eq_1}
\frac{1}{n}\sum_{v \in V}|f([G,v]_h)-f([G^{\pi},v]_h)|\leq \frac{1}{1+h}+\epsilon+\frac{2}{n}\sum_{v \in V}\mathbf{1}_{\{v \notin \Gamma\}}.
\end{equation}
In order to prove the bound in~\eqref{eq:disc-cont_eq_1} we split the expression on its left-hand side depending on whether  $v\in \Gamma$ or not. If $v\in \Gamma$, then $\d_{\Theta}(\tau_{w},\tau_{w}^{\pi}) \leq \epsilon$, for all $w \in V([G,v]_h)$, and $\d_{\Xi^2}(\xi_{e},\xi_{e}^{\pi})\leq \epsilon$, for all $e \in E([G,v]_h)$. As the underlying graphs of $[G,v]_h$ and $[G^{\pi},v]_h$ coincide it follows from the definition of the metric in $\bar{\mathcal{G}}^*$ that 
\begin{equation*}
\d_{\bar{\mathcal{G}}^*}([G,v]_h,[G^{\pi},v]_h)\leq \frac{1}{1+h}+\epsilon\,,\,\forall\,v \in \Gamma.
\end{equation*}
In particular, recalling that $f$ is $1$-Lipschitz,
\begin{equation*}
|f([G,v]_h)-f([G^{\pi},v]_h)|\leq \frac{1}{1+h}+\epsilon\,,\,\forall\,v\in\Gamma.
\end{equation*}
On the other hand, when $v \notin \Gamma$ we use that $f$ is bounded by $1$ to trivially estimate
\begin{equation*}
|f([G,v]_h)-f([G^{\pi},v]_h)|\leq 2.
\end{equation*}

\textit{Second step.} For fixed $S\in\mathbb{N}$ we define $\Delta=\{v\in V:|[G,v]_{2h}|\leq S\}$ and show that
\begin{equation}\label{eq:Sgamma}
\sum_{v \in V}\mathbf{1}_{\{v\notin \Gamma\}} \leq \sum_{v\in V}\mathbf{1}_{\{v\notin \Delta\}}+S\sum_{v \in V}\mathbf{1}_{\{\tau_{v} \in A^{*}\}}+S\sum_{e \in E}\mathbf{1}_{\{\xi_{e} \in B^{*}\}}.
\end{equation}

As before we split the left-hand side of~\eqref{eq:Sgamma} on whether $v\in \Delta$ or not. The crucial analysis is when $v\in \Delta$, i.e.,  $|[G,v]_{2h}|\leq S$. Since we are also assuming that $v\notin \Gamma$ then  either there exists  $w\in V([G,v]_{h})$ with $\tau_w\in A^*$ or there exists $e\in E([G,v]_h)$ with $\xi_e\in B^*$: 
\begin{equation}\label{eq:2sums}
\begin{split}
\sum_{v\in V}\mathbf{1}_{\{v\in \Delta,\,v\notin \Gamma\}}=&\sum_{v\in V}\sum_{w\in V}\mathbf{1}_{\{|[G,v]_{2h}|\leq S,\,w\in [G,v]_h,\,\tau_w\in A^* \}}\\
&+\sum_{v\in V}\sum_{e\in E}\mathbf{1}_{\{|[G,v]_{2h}|\leq S,\,e\in E([G,v]_h),\,\xi_e\in B^* \}}.
\end{split}
\end{equation}
We finish the proof by showing the following two estimates
\begin{equation}\label{eq:2sumsA}
\begin{split}
E_1:=\sum_{v\in V}\sum_{w\in V}\mathbf{1}_{\{|[G,v]_{2h}|\leq S,\,w\in [G,v]_h,\,\tau_w\in A^* \}}&\leq S\sum_{w\in V}\mathbf{1}_{\{\tau_w\in A^*\}}\,,\\
E_2:=\sum_{v\in V}\sum_{e\in E}\mathbf{1}_{\{|[G,v]_{2h}|\leq S,\,e\in E([G,v]_h),\,\xi_e\in B^* \}}&\leq S\sum_{e\in E}\mathbf{1}_{\{\xi_e\in B^*\}}.
\end{split}
\end{equation}
Let us start with the first bound in~\eqref{eq:2sumsA}. The goal is to keep only the sum on $w$. We will change the order of summation interchanging the role of $v$ and $w$. Observe that, if $w \in V([G,v]_h)$, then $v\in V([G,w]_h)$ and $V([G,w]_{h}) \subset V([G,v]_{2h})$. In particular, $v\in \Delta$ implies that $|V([G,w]_{h})| \leq |V([G,v]_{2h})|\leq S$. Therefore,
\begin{equation}\label{eq:Svertices}
\begin{split}
E_1&\leq\sum_{v\in V}\sum_{w\in V}\mathbf{1}_{\{|[G,w]_{h}|\leq S,\,\tau_w\in A^* \}}\mathbf{1}_{\{v\in [G,w]_h\}}\\
&=\sum_{w\in V}|[G,w]_{h}|\mathbf{1}_{\{|[G,w]_{h}|\leq S,\,\tau_w\in A^* \}}\\
&\leq S\sum_{w\in V}\mathbf{1}_{\{\tau_w\in A^*\}}.
\end{split}
\end{equation} 

The analysis of $E_2$ is very similar. For an edge $e\in E$, denote by $v_{e,1}$ and $v_{e,2}$ its endpoints: $e=\{v_{e,1},v_{e,2}\}.$ We define $B(e,h)=V([G,v_{e,1}]_h)\cup V([G,v_{e,2}]_h).$ Observe that $e\in E([G,v]_h)$ implies $v\in B(e,h)$ and $B(e,h)\subset [G,v]_{2h}.$ In particular, $|B(e,h)|\leq S$ if $v\in \Delta.$ The proof is concluded by repeating the same steps as in~\eqref{eq:Svertices}.

\section{Large deviation principle for discrete marks}\label{sec:discrete_marks}

\par We devote this section to the proof of Theorems~\ref{t:ldp_uniformg} and~\ref{t:ldp_erg}, which regard the case when the mark spaces $\Theta$ and $\Xi$ are finite metric spaces.

  Theorem~\ref{t:ldp} establishes a large deviation principle for the model of the uniform marked graph with prescribed vertex and edge mark count vectors. We follow~\cite{bc} showing that the models considered in Theorems~\ref{t:ldp_uniformg} and~\ref{t:ldp_erg} are mixture of the the DA model.

\subsection{Uniform graph with i.i.d.\ marks}
\label{sec:uniformgraph}
\par  We first show that the model of the uniform graph with i.i.d.\ marks (cf. Section~\ref{sec:model}) is written as mixture of the DA model. After that we deduce the large deviation principle from the mixture technique summarized in Appendix \ref{appendix:mixture_and_LDP}.

We assume throughout this section that $\Theta$ and $\Xi$ are finite sets.

In the following we calculate explicitly the distribution of $G_n$. First, for a pair $(x,y)\in \Xi\times\Xi$  we write $(x,y)_\leq=(x,y)$ if $x\leq y$ and $(x,y)_\leq=(y,x)$ otherwise.  For any $(H,\vec{\tau},\vec{\xi})\in \mathcal{G}_{\vec{m},\vec{u}}$ with $\|\vec{m}^{(n)}\|=m_n$ we have that
\begin{multline}\label{eq:lawiid}
\P[G_n=(H,\vec{\tau},\vec{\xi})]=\frac{1}{|\mathcal{G}_{n,m_n}|}\P(O_v= \tau_v\,,\forall\,v\in G_n)\\
\times\P((X_{\{u,v\}})_\leq=(\xi(u,v),\xi(v,u))_\leq\,,\forall \{u,v\}\in G_n) \frac{1}{2^{\sum_{x<\bar x}m(x,\bar x)}}\,.
\end{multline}

In order to deduce a large deviation principle for the sequence $(U(G_n))_{n\in\mathbb{N}}$ we write the distribution of $G_{n}$ as a mixture of the law of the DA model $G_{\vec{m},\vec{u}}$ when $(\vec{m}_\leq^{(n)}/m_n,\vec{u}_n/n)$ is sampled accordingly to $(L_n(\vec O),L_{m_n}(\vec X_\leq))$. 

\begin{lemma}\label{le:themix} For a fixed marked graph $\tilde{G}$, it holds that
\begin{equation}\label{eq:uufmixuif}
\P(G_n=\tilde{G})= \P(G_{\vec{m},\vec{u}}=\tilde{G})\P\bigg( L_{m_n}((\vec X)_\leq)=\frac{\vec m_\leq^{(n)}}{m_n}\bigg)\P\bigg( L_n(\vec O)=\frac{\vec u^{(n)}}{n} \bigg).
\end{equation}
\end{lemma}

Observe that the right-hand side of Equation~\eqref{eq:uufmixuif} is zero unless $\vec{m}^{(n)}=\vec{m}_{\tilde{G}}$ and $\vec{u}^{(n)}=\vec{u}_{\tilde{G}}$. We prove Lemma~\ref{le:themix} in Subsection~\ref{sec:themix}. Assuming its validity, we proceed to prove Theorem~\ref{t:ldp_uniformg}.

Consider $(P,Q)\in \mathcal{P}(\Xi^2_\leq)\times\mathcal{P}(\Theta)$ and  any sequence $(\vec{m}_n,\vec{u}_n)$ such that $\|\vec{m}_n\|_1=m_n$ and  $(\vec{m}_\leq^{(n)}/m_n,\vec{u}^{(n)}_n/n)$ converges weakly to $(P,Q)$. From Corollary~\ref{c:ldp} the empirical distribution of the DA model  $U(G_{\vec{m},\vec{u}})$ satisfies a large deviation principle with rate function given by
\begin{equation}\label{eq:rate_function_uniform_marked_graph}
I_{P,Q}(\mu)= H(Q)+s(\vec{d}(P,d)\,)-\Sigma_{\vec d(P,d),Q} (\mu).
\end{equation}

For $\mu \in \mathcal{P}(\bar{\mathcal{G}}_*)$, define $P(\vec{\deg}(\mu))$ via~\eqref{def:Pd} and
\begin{equation}
\lambda_{d}(\mu)= I_{P(\vec{\deg}(\mu)), \Pi(\mu)}(\mu) + \frac{d}{2}H \big( P(\vec{\deg}(\mu)) \big| \chi_\leq \big) + H(\Pi(\mu)|\nu)\,,
\end{equation}
where $\chi_\leq(x,x')=(\chi(x,x')+\chi(x',x))/2$ for $x\leq x'$.

Combining Corollary~\ref{c:ldp} with \eqref{eq:uufmixuif} we proceed to prove  Theorem~\ref{t:ldp_uniformg}.

\begin{proof}[Proof of Theorem~\ref{t:ldp_uniformg}]
Throughout this proof, we follow the notation from Appendix~\ref{appendix:mixture_and_LDP}. Set first $\Gamma = \mathcal{P}\big( \Xi^2_\leq \big)\times\mathcal{P}(\Theta) $ and consider $\mu_{n}$ as the distribution of the pair $\left(L_{m_n}(\vec X_\leq) ,L_n(\vec O)\right)$ with independent coordinates.
Sanov's Theorem implies that the sequence of measures $(\mu_n)_n$ satisfies a large deviation principle with speed $n$ and rate function
\begin{equation}
\psi(\alpha, \theta) = \frac{d}{2}H( \alpha | \chi_\leq) + H( \theta | \nu).
\end{equation}
Exponential tightness comes from the fact that $\Gamma$ is a compact metric space, since both $\Xi$ and $\Theta$ are assumed finite.

Second, let $\mathfrak{X}=\mathcal{P}(\bar{\mathcal{G}}_{*})$ and
\begin{equation}
\Gamma_{n} = \Big\{ \Big( \tfrac{\vec{m}_\leq^{(n)}}{m_{n}}, \tfrac{\vec{u}^{(n)}}{n} \Big) : \|\vec m^{(n)}\|_{1} = m_{n} \text{ and } \|\vec u^{(n)}\|_{1} = n \Big\} \subset \Gamma.
\end{equation}
Consider now a sequence $(\vec{m}^{(n)},u^{(n)})$ such that 
\begin{equation}\label{eq:converge}
\Big( \tfrac{\vec{m}_\leq^{(n)}}{m_{n}}, \tfrac{u^{(n)}}{n} \Big) \to \big( P, R \big)\in \Gamma\,.
\end{equation}
From Corollary~\ref{c:ldp} it follows that $U(G_{\vec{m},\vec{u}})$ is exponentially tight and satisfies a large deviation principle with speed $n$ and rate function $I_{P,R}$.

From Theorem~\ref{t:mixture}, $(U(G_n))_{n\in\mathbb{N}}$ satisfies a large deviation principle with speed $n$ and rate function defined by, for any $\mu \in \mathcal{P}(\bar G_*)$, 
\begin{equation*}
J_{d}(\mu)=\inf\bigg\{I_{P,R}(\mu)+\frac{d}{2}H(P |\chi_\leq) + H(R |\nu)\,:\,\big(P, R \big)\in \mathcal{P}(\Xi^2_\leq) \times \mathcal{P}(\Theta) \bigg\}.
\end{equation*}
Observe that $I_{P, R}(\mu)$ is infinite when either $P(\vec{\deg}(\mu)) \neq P$ or $\Pi_{\theta}(\mu) \neq R$, which in particular implies that  $I_{P,R}(\mu)$ is infinite if $\deg(\mu) \neq d$. As a consequence, we obtain the expression
\begin{equation*}
J_{d}(\mu)= \begin{cases} \lambda_d(\mu), & \text{ if } \deg(\mu) = d, \\
\infty, & \text{ otherwise,}
\end{cases}
\end{equation*}
and this concludes the proof.
 \end{proof}

\subsubsection{The mixture}\label{sec:themix}

In this section we prove Lemma~\ref{le:themix}. Recall the definition of $G_n$ from~\eqref{eq:lawiid} and that $\vec{O}$ and $\vec{X}$ have product laws.

In order to prove \eqref{eq:uufmixuif} we need to relate $\P(\vec{O}=\vec{\theta})$ to $\P(L_n(\vec{O})=\nu)$ and $|\mathcal{G}_{n,m_n}|$ to $|\mathcal{G}_{\vec{m},\vec{u}}|.$  First we introduce some additional notation based on \cite{dembo_zeitouni}. Let
\begin{equation*}
\mathcal{L}_n(\Theta)=\{L_n(\vec{\theta})\,:\,\vec{\theta}\in \Theta^n\}
\end{equation*}
be the set of all possible types and, for a measure $\nu \in \mathcal{L}_n(\Theta)$, we write
\begin{equation*}
T_n(\nu,\Theta)=\{\vec{\theta}\in \Theta^n\,:\,L_n(\vec{\theta})=\nu\}
\end{equation*}
for the type class of $\nu$. We have similar definitions for $\Xi^2_\leq$ instead of $\Theta$ and with $m_n$ instead of $n$.

\par 
From~\cite[Equation~9]{da} it holds that
\begin{equation}\label{eq:indianos-types}
|\mathcal{G}^{(n)}_{\vec{m},\vec{u}}|=|\mathcal{G}_{n,m_n}|\cdot\bigg|T_n\bigg(\frac{\vec{u}^{(n)}}{n},\Theta\bigg)\bigg|\cdot\bigg|T_{m_n}\bigg(\frac{\vec{m}_\leq^{(n)}}{m_n},\Xi^2_\leq\bigg)\bigg|\cdot 2^{\sum_{x<\bar x}m(x,\bar x)}\,.
\end{equation}

Lemma~\ref{le:themix} is deduced by combining Proposition~\ref{prop:exch} below and Equation~\eqref{eq:indianos-types}.

\begin{proposition}\label{prop:exch} Let $\vec{\theta}\in \Theta^n$ and $\vec{\xi}\in (\Xi^2_\leq)^{m_n}$. For any $\nu=L_n(\vec{\theta})$ and $\lambda=L_{m_n}(\vec{\xi})$ we have 
\begin{align*}
\P(\vec{O}=\vec{\theta})=\frac{\P(L_n(\vec{O})=\nu)}{|T_n(\nu,\Theta)|} \,\,\text{ and }\,\, \P(\vec{X}_\leq=\vec{\xi})=\frac{\P(L_{m_n}(\vec{X}_\leq)=\lambda)}{|T_{m_n}(\lambda,\Xi^2_\leq)|}.
\end{align*}
\end{proposition}

\begin{proof}
We only prove the first formula and the second one is analogous.
Observe that 
\begin{equation*}
    \P(L_n(\vec O)=\nu) = \P(\vec O \in T_n(\nu, \Theta)) = \sum_{ \vec{y} \in T_n(\nu, \Theta)} \P(\vec{O} = \vec{y}).
\end{equation*}
Now, if  $\vec{y}\in T_n(\nu,\Theta)$, then $L_n(\vec y)=\nu = L_n(\theta)$. Thus, there exists a permutation $\sigma^{\vec y} \in S_n$ such that $y_i=\theta_{\sigma^{\vec y}(i)}$. Since $\vec{O}$ has product law, it is exchangeable and we obtain
\begin{equation*}
\P(\vec O = \vec{y}) = \P(\vec{O}=\vec{\theta}),
\end{equation*}
which yields
\begin{equation*}
\P(L_n(\vec O)=\nu) = |T_n(\nu,\Theta)| \cdot \P(\vec O = \vec{\theta}),
\end{equation*}
concluding the proof.
\end{proof}

\subsection{Erd\H{o}s-R\'enyi random graph with i.i.d.\ marks}\label{subset:er}

\par The goal of this section is to prove Theorem~\ref{t:ldp_erg} assuming Theorem~\ref{t:ldp_uniformg}. The idea of the proof is to use that the Erd\H{o}s-R\'enyi graph is equal in distribution to the uniform graph when the number of edges is chosen accordingly to a Binomial distribution. This mixture structure extends naturally when the graphs have i.i.d.\ marks and also for their empirical neighborhood distribution. The main technicality of this section is to apply the results from Appendix~\ref{appendix:mixture_and_LDP}.

Let $G^d_{n}$ be the random marked graph with i.i.d.\ marks when the underlying graph is sampled by $G(n,d/n)$. Recall that for a pair $(x,y)\in \Xi\times\Xi$  we write $(x,y)_\leq=(x,y)$ if $x\leq y$ and $(x,y)_\leq=(y,x)$ otherwise.  For any $\tilde{G}=(H,\vec{\tau},\vec{\xi})\in \mathcal{G}_{\vec{m},\vec{u}}$ with $\|\vec{m}^{(n)}\|=m_n$ we have by definition that
\begin{multline}\label{eq:lawiidER}
\P[G^{d}_n =(H,\vec{\tau},\vec{\xi})]=  \left( \frac{d}{n} \right)^{m_n} \left(1-\frac{d}{n} \right)^{\binom{n}{2}-m_n} \P(O_v= \tau_v\,,\forall\,v\in G_n)\\
\times\P((X_{\{u,v\}})_\leq=(\xi(u,v),\xi(v,u))_\leq\,,\forall \{u,v\}\in G_n)\frac{1}{2^{\sum_{x<\bar x}m(x,\bar x)}}\,.
\end{multline}

In Section~\ref{sec:uniformgraph} we considered a sequence of uniform graphs with fixed number $(m_n)_{n\in \mathbb{N}}$ of edges.  In this section it is convenient to consider a model of uniform graph where the number of edges is itself random. Motivated by that, we denote by $G_{n,m}$ the a marked graph with i.i.d. marks when the underlying graph is sampled from $\mathcal{G}_{n,m}$. Let $M_n$ be a random variable sampled accordingly to $ \textrm{Bin}\left( \binom{n}{2}, d/n \right)$.

\begin{lemma}
$G^{d}_n$ has the same distribution as $G_{n,M_n}$, i.e., we have the following mixture structure
\begin{align*}
    \mathbb{P}(G^{d}_n = G) = \sum_{m=1}^{n(n-1)/2} \mathbb{P} ( G_{n,m} = G ) \cdot \mathbb{P} ( M_n = m ) .
\end{align*}
\end{lemma}

\begin{proof}
Let $(H,\vec{\tau},\vec{\xi}) \in \mathcal{G}_{\vec{m},\vec{u}}$ and $\|\vec{m}^{(n)}\|=m_n$. In Equation \eqref{eq:lawiidER} we have the distribution of $G^{d}_n$. If we multiply and divide the right-hand side by $|\mathcal{G}_{n,m_n}|$ we obtain
\begin{multline}
\P[G^{d}_n = (H,\vec{\tau},\vec{\xi})]=   |\mathcal{G}_{n,m_n}| \left( \frac{d}{n} \right)^{m_n} \left(1-\frac{d}{n} \right)^{\binom{n}{2}-m_n} \P(O_v= \tau_v\,,\forall\,v\in G_n)\\
\times 
    \frac{1}{|\mathcal{G}_{n,m_n}|} \: \P((X_{\{u,v\}})_\leq=(\xi(u,v),\xi(v,u))_\leq\,,\forall \{u,v\}\in G_n) \frac{1}{2^{\sum_{x<\bar x}m(x,\bar x)}}\,.
\end{multline}

But, observe that $ |\mathcal{G}_{n,m_n}| \left( \frac{d}{n} \right)^{m_n} (1-\frac{d}{n})^{\binom{n}{2}-m_n}$ is exactly the probability that $M_n=m_n$. Therefore,
\begin{align*}
    \mathbb{P} \big( G^{d}_n = (H,\vec{\tau},\vec{\xi}) \big) = \mathbb{P} ( M_n = m_n ) \mathbb{P} ( G_{n,m_n} = (H,\vec{\tau},\vec{\xi}) ).
\end{align*}
By convention,  $ \mathbb{P} ( G_{n,m} = G ) = 0$ if $m \neq m_G$ and the final formula holds.
\end{proof}

The result above, in particular, gives us the mixture structure
\begin{equation}
    \mathbb{P}(U(G^{d}_n) \in A) = \int_{\gamma \in \Gamma_n} \mathbb{P}( U(G_{n, \frac{1}{2}\gamma n}) \in A) \: d \mu^n(\gamma),
\end{equation}
where $\Gamma_n = \{0,2/n, 4/n, \dots, n(n-1)/n\}$ and $\mu^n$ is the law of $2M_n/n$.

To obtain the desired large deviation principle we use Theorem~\ref{t:mixture}. The first step to apply that theorem is to check that $\mu_n$ satisfies a large deviation principle. Indeed, from Gartner-Ellis Theorem we have the following result.
\begin{lemma}
The sequence $(\mu_n)_{n\in\mathbb{N}}$ satisfies a large deviation principle in the space $[0, \infty)$ with speed $n$ and good rate function 
\[ \psi(x) = \frac{1}{2} \Big( x \log \frac{x}{d}- x +d \Big)   .\]
\end{lemma}

In particular, the sequence $\mu_n$ is exponentially tight and it satisfies the $\mu$-assumptions (cf. Appendix~\ref{appendix:mixture_and_LDP}).

The next step is to verify that the $\Gamma$-assumptions hold. Following the notation of Appendix~\ref{appendix:mixture_and_LDP} we set
\begin{equation*}
    P^{n}_{\gamma_n}(\,\cdot\,) =  \mathbb{P}( U(G_{n, \frac{1}{2}\gamma_n n}) \in \,\cdot\,)\,.
\end{equation*}
Let 
\begin{equation*}
    \Tilde{\Gamma} = \{ \gamma \in \Gamma : \: \exists\, \gamma_n \in \Gamma_n \: \textrm{with} \: \gamma_n \to \gamma\}.
\end{equation*}
 From Theorem~\ref{t:ldp_uniformg}, if $\gamma_n \to \gamma$, the sequence of random variables $U(G_{n, \frac{1}{2}\gamma_n n})$ satisfies a large deviation principle with rate function $I_\gamma$ defined in~\eqref{eq:rate_function_uniform_graph_2}.

As consequence of Theorem~\ref{t:mixture}, the sequence of empirical neighborhood distributions $(U(G_n^{d}))_{n\in\mathbb{N}}$ satisfies a large deviation principle with rate function given by
\begin{equation}
\begin{split}
     \lambda(\mu)& = \inf\{ I_{\gamma}(\mu) + \psi(\gamma) : \gamma \in \Gamma\}\\
     &= I_{\deg(\mu)}(\mu) + \psi(\deg(\mu))\,,
\end{split}
\end{equation}
since $I_{\gamma}(\mu) = \infty$ if $\gamma\neq\deg(\mu)$.

\section{Determining the minimizers}\label{sec:minimizers}

\par In this section, we examine the minimizers of the rate functions. We start by examining minimizers of the rate function for the DA model and obtain expressions for the progressively more complicated models.

\subsection{Minimizers for $I_{\vec{d},Q}$}\label{subsec:existence}

\par In this section we verify Proposition~\ref{prop:minimizer}, the main step towards determining minimizers for the other rate functions obtained.

Recall that a minimizer of the rate function is a probability distribution for which $I_{\vec{d},Q}$ attains the value zero. By the definition of $I_{\vec{d},Q}$ (see Equation~\eqref{eq:I_d_Q}), $\PP^{*}_{d}$ will be a minimizer if $H(Q) +s(\vec{d}) = \Sigma_{\vec{d},Q}(\PP^{*}_{d})$. In view of Proposition~\ref{prop:entropyUGW}, it suffices to prove that $J_1(\PP^*_{d})=H(Q)+s(\vec d)$. This is what we in fact verify in the proof of Proposition~\ref{prop:minimizer}.

\begin{proof}[Proof of Proposition~\ref{prop:minimizer}]
As we already mentioned, Proposition~\ref{prop:minimizer} follows from
\begin{equation}\label{eq:P*minimum}
J_1(\PP^*_{d})=H(Q)+s(\vec d).
\end{equation} 
Recall $E_h (t,t' )$ from~\eqref{eq:E_1}. Equation~\eqref{eq:J_h} yields
\begin{equation*}
J_1(\PP^*_{d})= -s(d) + H(\PP^*_{d}) -\frac{d}{2}H(\pi_{\PP^*_{d}} )-\sum_{t,t' \in \Xi\times\bar{\mathcal{T}}^{0}_*}\EE_{\PP^{*}_{d}}(\log E_h (t,t' )!).
\end{equation*}
Thus, our claim follows if we verify
\begin{align}
-\frac{d}{2}H(\pi_{\PP^*_{d}})&=-dH(Q)-s(\vec{d})+s(d),\label{first:id}\\
H(\PP^*_{d})-\sum_{t,t' \in \Xi\times\bar{\mathcal{T}}^{0}_*}
\EE_{\PP^*_{d}}(\log E_1 (t,t' )!)&=H(Q)+dH(Q)+2s(\vec{d}).\label{second:id}
\end{align}

We first check~\eqref{first:id}. Recall that we identify the set $\mathcal{T}^{0}_{*}$ with $\Theta$ and the definition of $\pi_{\PP^{*}_{d}}$ in~\eqref{eq:pi_P}. From the definition of $\PP^*_{d}$ we obtain
\begin{align*}
-\frac{d}{2}H(\pi_{\PP^*_{d}})= & \sum_{\substack{\theta,\theta'\\x,x'}}Q(\theta)Q(\theta')d_{x,x'}\log Q(\theta)+\frac{1}{2}\sum_{\substack{\theta,\theta'\\x,x'}}Q(\theta)Q(\theta')d_{x,x'}\log d_{x,x'}\\
& -\frac{1}{2}\sum_{\substack{\theta,\theta'\\x,x'}}Q(\theta)Q(\theta')d_{x,x'}\log d.
\end{align*}
The definition of $s(\vec{d})$ and $s(d)$ in Equation~\eqref{def:s} and the fact that $\sum_{x,x'} d_{x,x'}=d$ implies~\eqref{first:id}.

We now aim to check~\eqref{second:id}. We first write $H(\PP^*_{d})$ as the sum of $H(Q)$ with another expression. The first step is to split the set of rooted trees accordingly to the mark of the root:
\begin{align*}
H(\PP^*_{d})=&-\sum_\theta\sum_{t:\theta_t=\theta}\PP^*_{d}(t)\log \Big( Q(\theta)\prod_{\theta',x,x'}{\rm Poi}(\alpha_{x,x'}^{\theta'})(N_{x,x'}^{\theta,\theta'}(t))\Big)\\
=&H(Q)-\sum_\theta\sum_{t:\theta_t=\theta}\PP^*_{d}(t)\sum_{\theta',x,x'}\log {\rm Poi}(\alpha_{x,x'}^{\theta'})(N_{x,x'}^{\theta_t,\theta'}(t)).
\end{align*}

We also have,
\begin{equation*}
-\sum_{\tilde{t},t' \in \Xi\times\bar{\mathcal{T}}^{0}_*}\EE_{\PP^*_{d}}(\log E_1 (\tilde{t},t' )!)=-\sum_\theta
\sum_{t:\theta_t=\theta}\PP^*_{d}(t)\sum_{\substack{\theta',x,x'}}\log \Big( N_{x,x'}^{\theta,\theta'}(t)! \Big).
\end{equation*}
Therefore,
\begin{align*}
H(\PP^*_{d}) &-
\sum_{t,t' \in \Xi\times\bar{\mathcal{T}}^{0}_*} \EE_{\PP^*_{d}}(\log E_1 (t,t' )!)=\\
&=H(Q)-\sum_\theta
\sum_{t:\theta_t=\theta}\PP^*_{d}(t)\sum_{\substack{\theta',x,x'}}\log \Big( {\rm Poi}(\alpha_{x,x'}^{\theta'})(N_{x,x'}^{\theta,\theta'}(t)) \cdot N_{x,x'}^{\theta,\theta'}(t)! \Big)\\
&=H(Q)-\sum_\theta
\sum_{t:\theta_t=\theta}\PP^*_{d}(t)\sum_{\substack{\theta',x,x'}}\bigg(-\alpha_{x,x'}^{\theta'}+N_{x,x'}^{\theta,\theta'}(t)\log \alpha_{x,x'}^{\theta'}\bigg).
\end{align*}
We now use the fact that $d=\sum_{x,x'} d_{x,x'}$ to obtain
\begin{align*}
H(\PP^*_{d}) - \sum_{t,t' \in \Xi\times\bar{\mathcal{T}}^{0}_*} & \EE_{\PP^*_{d}}(\log E_1 (t,t' )!) \\
& =H(Q)+d-\sum_{\substack{\theta,\theta'\\x,x'}}Q(\theta)Q(\theta')d_{x,x'}\log Q(\theta')d_{x,x'}\\
& =H(Q)+d+dH(Q)-\sum_{x,x'}d_{x,x'}\log d_{x,x'}\\
& =H(Q)+dH(Q)+2s(\vec{d}).
\end{align*}
This concludes the verification of~\eqref{second:id} and implies that the marked unimodular Galton-Watson with root $\PP^{*}_{d}$ is a minimizer of the rate function $I_{\vec{d},Q}$ from Theorem~\ref{t:ldp}.
\end{proof}

\subsection{Minimizers for the discrete case}\label{subsec:minimizers_discrete}

\par We are now in position to establish the expression of the minimizer for the rate functions of the large deviation principle satisfied by the marked uniform random graph and the marked sparse Erd\H{o}s-R\'enyi random graph when the spaces of marks $\Theta$ and $\Xi$ are finite. This case is rather simple, since we obtain explicit expressions for the rate functions.

\begin{proof}[Proof of Proposition~\ref{prop:minimizer_discrete}]
Let us first verify that the unimodular Galton-Watson tree with seed $\PP^{*}_{\nu, \chi, d}$ is a minimizer of the rate functions $I^{u}_{d}$. Recall from~\eqref{eq:rate_function_uniform_graph_2} that $I^{u}_{d}(\mu)$ is finite only if $\deg(\mu)=d$ and that, in this case,
\begin{equation}
I^{u}_{d}(\mu)=I_{P(\vec{\deg}(\mu)), \Pi(\mu)}(\mu) + \frac{d}{2}H \big( P(\vec{\deg}(\mu)) \big| \chi_\leq \big) + H(\Pi(\mu)|\nu).
\end{equation}
From this we see that the distribution of the vertex marks should be $\nu$ and that the distribution of the (non-oriented) edge marks should be $\chi_{\leq}$ in order to obtain a minimizer. This last requirement is verified if we assume that marks are attributed with distribution $\chi$. Furthermore, Proposition~\ref{prop:minimizer} yields that the unimodular Galton-Watson tree with seed $\PP^{*}_{\nu, \chi, d}$ is a minimizer of the rate function $I_{\chi_{\leq},\, \nu}$. This concludes the proof of the first case.

In the case of the sparse Erd\H{o}s-R\'enyi random graph, the rate function is given by
\begin{equation}
I_{d}^{\rm ER}(\mu)=I^u_{\deg(\mu)}(\mu)+\frac{1}{2}\Big( \deg(\mu) \log \frac{\deg(\mu)}{d}- \deg(\mu) +d \Big).
\end{equation}
From this, any minimizer should satisfy $\deg(\mu)=d$. The result then follows from the case of the uniform random graph.
\end{proof}

\subsection{Minimizers for the general case}\label{subsec:minimizers_general}

\par We finalize this section considering the general case when the spaces $\Theta$ and $\Xi$ are Polish spaces. Notice that in this case we obtain the rate function through approximations and we use this fact to establish the existence of minimizers.

\begin{proof}[Proof of Proposition~\ref{prop:minimizer_general}]
In the case of the uniform random graph, the rate function is given by
\begin{equation}
\bar{I}_d^{u}(\mu) = \limsup_{\delta \to 0} \limsup_{k \to \infty} \inf_{ \nu \in B_{\delta}(\mu) } I_d^{u,k}(\nu).
\end{equation}

Furthermore, since the marked unimodular Galton-Watson tree with seed $\PP^{*}_{\nu^{k}, \chi^{k}, d}$ is a minimizer of the rate functions $I^{u,k}_{d}$, it suffices to verify that
\begin{equation}
{\rm UGWT}(\PP^{*}_{\nu^{k}, \chi^{k}, d}) \to {\rm UGWT}(\PP^{*}_{\nu, \chi, d}),
\end{equation}
as $k$ grows. The same applies for the minimizer of $\bar{I}_d^{\rm ER}$.

Let us now verify the convergence above. Let $t$ be a random rooted tree distributed as ${\rm UGWT}(\PP^{*}_{\nu, \chi, d})$ and obtain $t_{k}$ with distribution $ {\rm UGWT}(\PP^{*}_{\nu^{k}, \chi^{k}, d})$ from $t$ by projecting the marks on edges and vertices. We have
\begin{equation}
\d_{BL}( {\rm UGWT}(\PP^{*}_{\nu^{k}, \chi^{k}, d}), {\rm UGWT}(\PP^{*}_{\nu, \chi, d})) \leq \E \big( \d_{\bar{\mathcal{G}}_{*}}(t, t_{k}) \wedge 2 \big).
\end{equation}
We split in the cases where the distance is smaller than $2/k$ or not. When the distance is not smaller we bound the expectation by 2 times the respective probability. 
The proof will then be completed if we prove that
\begin{equation}\label{eq:limit_equals_zero}
\lim_{k} \P\Big( \d_{\bar{\mathcal{G}}_{*}}(t, t_{k}) > \frac{2}{k} \Big) =0.
\end{equation}

In order to verify the equation above, notice that, if $\d_{\bar{\mathcal{G}}_{*}}(t, t_{k}) > \frac{2}{k}$ then there exists a vertex or an edge that is at distance at most $k$ from the root $o$ whose corresponding mark belongs to $A^{*}$ or $B^{*}$, respectively (recall these sets where introduced in Lemma~\ref{lemma:distace_estimate}). Consider then the event
\begin{equation}
E_{k} = \Big\{ \begin{array}{c}
\text{there exists a vertex or an edge that is at distance at most } k \\
\text{from the root } o \text{ with mark in } A^{*} \text { or } B^{*}
\end{array}
\Big\}.
\end{equation}
Union bound immediately implies
\begin{equation}\label{eq:estiamte_galton_watson}
\begin{split}
\P(E_{k}) & \leq \P \bigg( |[t,o]_{k}| \geq \sum_{\ell=0}^{k} k^{\ell} \bigg) + (\nu(A^{*})+\xi(B^{*}))\sum_{\ell=0}^{k} k^{\ell} \\
& \leq \sum_{\ell=1}^{k}\P \big( |[t,o]_{\ell} \setminus [t,o]_{\ell-1}| \geq k^{\ell} \big) + 2e^{-k^2}(k+1)k^{k} \\
& \leq \sum_{\ell=1}^{k} \frac{\E\big( |[t,o]_{\ell} \setminus [t,o]_{\ell-1}| \big)}{k^{\ell}} + 2e^{-k^2}(k+1)k^{k} \\
& = \sum_{\ell=1}^{k}\frac{d^{\ell}}{k^{\ell}} + 2e^{-k^2}(k+1)k^{k} \leq \frac{d}{k-d}+2e^{-k^2}(k+1)k^{k}
\end{split}
\end{equation}

Since the right-hand side of the expression above converges to zero as $k$ increases, we obtain~\eqref{eq:limit_equals_zero}. In particular, this implies, for each $\delta>0$,
\begin{equation}
\limsup_{k \to \infty} \inf_{ \nu \in B_{\delta}({\rm UGWT}(\PP^{*}_{\nu, \chi, d})) } I_d^{u,k}(\nu) =0,
\end{equation}
which implies that ${\rm UGWT}(\PP^{*}_{\nu, \chi, d})$ is indeed a minimizer of $\bar{I}_d^{u}$. The proof of this fact for $\bar{I}_d^{\rm ER}$ follows exactly the same lines.
\end{proof}

\section{Application - Large deviation principle for interacting diffusions on sparse marked graphs}
\label{sec:idg}

In this section we obtain a large deviation principle for interacting diffusions on sparse marked graphs, which provides an example of application of Theorems~\ref{t:ldp_uniformg_general} and~\ref{t:ldp_erg_general}.

We divide this section in the following steps:
\begin{enumerate}
\item We introduce the marked graphs with the elements necessary to define an interacting diffusion: initial conditions, media variables, strength of interactions.
\item We introduce the interacting diffusions in the special case of gradient evolution.
\item We  define the marked graph that includes these interacting diffusions.
\item We finish the section with the statement of the main theorem and the organization of the proof.
\end{enumerate}

\subsubsection{The network}\label{sec:intdiff:network}
Let $G=(V,E,\vec{\xi},\vec{\omega},\vec{\theta}(0))$ be a finite marked graph where:
\begin{itemize}
\item $\vec{\xi}=(\xi_{u,v};\,\{u,v\}\in E)\in \R^{2|E|}$ are marks on the oriented edges. We can interpret $\xi_{u,v}$ as the strength of interaction associated to the oriented edge $(u,v).$
\item $\vec{\omega}=(\omega_{v};\,v\in V)\in \R^{|V|}$ are marks on vertices that represent ``media" variables.
\item $\vec{\theta}(0)=(\theta_{v}(0);\,v\in V)\in \R^{|V|}$ are marks on vertices. They are the initial conditions of the interacting diffusion.
\end{itemize}

\subsubsection{The interacting diffusions}
Besides the marked graph $G=(V,E,\vec{\xi},\vec{\omega},\vec{\theta}(0))$, we fix two functions $f$ and $g$ that play the role of pair potential and external field. We define the Hamiltonian, which depends on $G$, by
\begin{equation}\label{eq:hamiltonian}
H_{G}(\vec{x})=\sum_{v,u\in V}\xi_{v,u}f(x_{v}-x_{u};\,\omega_v,\omega_u)+\sum_{v\in V}g(x_v;\omega_v).
\end{equation}
Observe that we do not normalize the Hamiltonian and this is compatible with our sparsity assumptions on $G$.
We write $\partial_{v}H_{G}$ for the derivative of $H_{G}$ with respect to the $x_{v}$ variable.
We fix a finite time horizon $T>0$ and let $(B_v;\,v\in V)$ be i.i.d. standard Brownian motions defined on the time interval $[0,T].$
The interacting diffusions over $G$ is the stochastic process $\vec{\theta}^{G}=(\theta^{G}_{v}(t);\,v\in V,\,t\in [0,T])$ which solves the following system of It\^o stochastic differential equations
\begin{equation}\label{eq:sde}
\begin{cases}
{\rm d}\theta^{G}_{v}(t)&=\partial_{v}H_G(\vec{\theta}^{G},\vec{\xi},\vec{\omega}){\rm d}t+ {\rm d}B_{v}(t),\,\,0\leq t\leq T,\\
\theta^{G}_{v}(0)&=\theta_{v}(0),\, \forall\,v\in V.
\end{cases}
\end{equation}

We assume that $f$, $f'$ (derivative with respect to the $x$ variable), $f''$, $g$, $g'$, $g''$ exist, are bounded and continuous. Under this assumptions, when the marked graph $G$ is finite, the system  (\ref{eq:sde}) has a unique strong solution with continuous trajectories (see \cite[Chapter~5, Theorem~2.9]{karatzas2012brownian}).

\subsubsection{The new network}

We include the solutions of the above system of SDE's as new marks of our original marked graph. For that reason we define $\overline{\mathcal{G}}_{*}^{0}$ for the space of marked graphs with the marks including $C([0,T];\R)$. More specifically, $\overline{\mathcal{G}}_{*}^{0}$ is the set of rooted marked graphs of the form
\begin{equation}\label{eq:mrg_diffusion}
(G\otimes \vec{x},o):=(V,E,\vec{\xi},\vec{\omega},\vec{\theta}(0),\vec{x}_T,o),
\end{equation}
where $\vec{x}_T=(x_v;\,x_v\in C([0,T];\R),\,v\in V)$ is a field of continuous functions indexed by the vertices of $G$.  As a particular example, consider $\vec{x}_T=\vec{\theta}^{G}_T$.

 Let $(G_n)_{n\in \N}$ be a sequence of finite random marked graphs of the form described in Section~\ref{sec:intdiff:network} with vertex set given by $[n]$. For each $n\in\mathbb{N}$ let $\vec{B}^{(n)}=(B^{(n)}_v)_{v\in [n]}$ be i.i.d.\ Brownian motion on the time interval $[0,T]$

\begin{theorem}\label{thm:intdiff} Assume that $(U(G_n\otimes \vec{B}^{(n)});\,n\in\N)$ satisfies a large deviation principle with rate function $I:\mathcal{P}(\overline{\mathcal{G}}_*^{0})\to [0,+\infty].$ Then the annealed law of the sequence $(U(G_n\otimes \vec{\theta}^{G_n});\,n\in \N)$ satisfy a large deviation principle with rate function $I(\rho)-F(\rho)$ where $F(\rho)$ is given by
\begin{equation}\label{def:F}
F(\rho)=-\int \bigg(F^{(1)}_T-F^{(1)}_0+\frac{1}{2}F^{(2)}+\frac{1}{2}F^{(3)}\bigg){\rm d}\rho\,,
\end{equation}
where the functions $F^{(1)}_t,F^{(2)},F^{(3)}:\overline{\mathcal{G}}_*^{0}\to \R$, $t\in \{0,T\}$, are defined by, for $G=(V,E,\vec{\xi},\vec{\omega},\vec{x},o)$,
\begin{equation}\label{def:Fi}
\begin{split}
F^{(1)}_t(G)=\sum_{w;w\sim o}& \xi_{o,u}f(x_{o}(t)-x_{u}(t);\,\omega_o,\omega_u)+g(x_o(t);\omega_o),\\
F^{(2)}(G)=\int_{0}^{T}\bigg[&\sum_{u;u\sim o}\xi_{ou}f'(x_o(t)-x_u(t);\omega_o,\omega_u)\\
&-\sum_{u;u\sim o}\xi_{uo}f'(x_u(t)-x_o(t);\omega_o,\omega_u)+g'(x_o(t);\omega_o)\bigg]^2{\rm d}t,\\
F^{(3)}(G)=\int_{0}^{T}\bigg[&\sum_{u;u\sim o}\xi_{ou}f''(x_o(t)-x_u(t);\omega_o,\omega_u)\\
&+\sum_{u;u\sim o}\xi_{uo}f''(x_u(t)-x_o(t);\omega_o,\omega_u)+g''(x_o(t);\omega_o)\bigg]{\rm d}t.
\end{split}
\end{equation}
\end{theorem}

This result is the analog of \cite[Theorem~1]{dpdh} which considers the mean-field case. The strategy of the proof is also similar: we write the law of $U(G_n\otimes \vec{\theta}^{G_n})$ as an exponential tilt with respect to the law of $(U(G_n\otimes \vec{B}^{(n)});\,n\in\N)$. The result follows from an application of Varadhan's Lemma assuming that $(U(G_n\otimes \vec{B}^{(n)});\,n\in\N)$ satisfies a large deviation principle. See the details in Section~\ref{sec:proof:intdiff}.

\begin{remark} Consider a sequence $G_n=(V_n,E_n,\vec{\xi}^{(n)},\vec{\omega}^{(n)},\vec{\theta}^{(n)}(0))$ where
\begin{enumerate}
\item $(V_n,E_n)$ is distributed either as the Erd\H{o}s-R\'enyi graph or the uniform random graph as in Section~\ref{sec:model}.
\item the components of $\vec{\xi}^{(n)}$, $\vec{\omega}^{(n)}$, and $\vec{\theta}^{(n)}(0)$ are i.i.d.\ random variables with fixed distributions.
\end{enumerate}
Theorems~\ref{t:ldp_uniformg_general} and~\ref{t:ldp_erg_general} imply that $(U(G_n\otimes \vec{B}^{(n)});\,n\in\N)$ satisfies a large deviation principle. This gives examples where Theorem~\ref{thm:intdiff} can be applied.
\end{remark}

\subsection{Proof of Theorem~\ref{thm:intdiff}}\label{sec:proof:intdiff}

The main preliminary result necessary to prove Theorem~\ref{thm:intdiff} is the following lemma that compares Radon-Nikodym derivatives of marked graphs with marks given by idependent Brownian motions and the solutions of the SDE~\eqref{eq:sde}.

Let $(G_n;\,n\in \N)$ be the sequence of random networks as in~\eqref{sec:intdiff:network} and assume that the vertex set of $G_n$ is $[n]$. We write $\PP_n^{G}$ to denote the law of the solutions $\vec{\theta}^{G_n}\in C([0,T];\R^{n})$ of~\eqref{eq:sde}, conditioned on the law of $G_n$. Write $W^{\otimes n}$ for the law of the i.i.d. standard Brownian motions $(B_v;\,v\in [n])\in C([0,T];\R^{n})$.

Following~\cite{dpdh}, Theorem~\ref{thm:intdiff} is a consequence of Varadhan's Lemma together with the following result.
\begin{lemma}\label{lem:radomnykodym}
Conditioned on $G_n$,
\begin{equation}\label{eq:exptilt}
\frac{\d \PP^G_n}{\d W^{\otimes n}}(G_n \otimes \vec{x})=\exp(nF(U(G_n\otimes \vec{x}))),
\end{equation}
where, for $\rho \in \mathcal{P}(\overline{\mathcal{G}}_{*}^{0})$, $F(\rho)$ is given by~\eqref{def:F}. Furthermore, under our assumptions on $f$ and $g$, the function $F$ is a bounded and continuous.
\end{lemma}

\begin{proof}[Proof of Lemma~\ref{lem:radomnykodym}]
In order to establish the lemma, it suffices to verify the following steps:
\begin{enumerate}
\item Calculate the Radon-Nikodym derivative $\frac{\d \PP_n^G}{\d W^{\otimes n}}$ using Girsanov's Theorem.
\item Obtain~\eqref{eq:exptilt} via It\^o's formula.
\end{enumerate} 

We write $H_n:=H_{G_n}$ for the Hamiltonian defined in \eqref{eq:hamiltonian}.
Aplying Girsanov's formula (see~\cite[Appendice A]{bpr} and~\cite[Equation 2.2]{dpdh}) we obtain
\begin{equation}
\frac{\d \PP^G_n}{\d W^{\otimes n}}(G_{n} \otimes \vec{x}_T)=\exp\bigg[- \frac{1}{\sigma^2}\sum_{v=1}^n\int_{0}^{T}\partial_vH_n(\vec{x}(t)) \d x_v(t) - \frac{1}{2\sigma^2}\sum_{v=1}^n\int_0^T\big(\partial_vH_n(\vec{x}(t))\big)^2 \d t\bigg].
\end{equation}
As consequence of It\^o's formula we write
\begin{equation}
\begin{split}
\sum_{v=1}^n\int_{0}^{T}\partial_v H_n(\vec{x}(t)) \d x_v(t)= H_n(\vec{x}(T))-H_n(\vec{x}(0))-\frac{1}{2}\sum_{v=1}^n\int_{0}^{T} \partial^{2}_{vv}H_n(\vec{x}(t)) \d t.
\end{split}
\end{equation}
Therefore,
\begin{equation}
\begin{split}
\frac{\d \PP^G_n}{\d W^{\otimes n}}(G_n\otimes \vec{x})& =\exp\bigg[-\frac{1}{\sigma^2}\big[H_n(\vec{x}(T)) - H_n(\vec{x}(0))\big] \\
& -\frac{1}{2\sigma^2}\sum_{v=1}^n \int_{0}^{T}\partial^{2}_{vv}H_n(\vec{x}(t)) \d t -\frac{1}{2\sigma^2}\sum_{v=1}^n\int_0^T\big(\partial_vH_n(\vec{x}(t))\big)^2 \d t \bigg].
\end{split}
\end{equation}

The next step is rewrite each term of the right-hand side of the above expression as a function of $U(G_n^{x}).$
To simplify notation, let's suppress the dependency in $n$ and on the variables $\xi$, $\omega$, $\vec{x}$.

First, we rewrite
\begin{equation}
\begin{split}
H_n(\vec{x}(T))-& H_n(\vec{x}(0))= \\
& \sum_{v=1}^{n}\sum_{u;u\sim v}\xi_{v,u}f(x_{v}(T)-x_{u}(T);\,\omega_v,\omega_u) +\sum_{v=1}^{n}g(x_v(T);\omega_v)\\
& + \sum_{v=1}^{n}\sum_{u;u\sim v}\xi_{v,u}f(x_{v}(0)-x_{u}(0);\,\omega_v,\omega_u)+\sum_{v=1}^{n}g(x_v(0);\omega_v).
\end{split}
\end{equation}

As direct consequence of \eqref{def:Fi}:
\begin{equation}
\begin{split}
H_n(\vec{x}(T))-H_n(\vec{x}(0))=n\int (F^{(1)}_T-F^{(1)}_0) \d U(G_n).
\end{split}
\end{equation}
A simple computation leads to
\begin{equation}
\begin{split}
\partial_{v}H_n(\vec{x}(t))=&\sum_{u;u\sim v}\xi_{vu}f'(x_v(t)-x_u(t);\omega_v,\omega_u)\\
&-\sum_{u;u\sim v}\xi_{uv}f'(x_u(t)-x_v(t);\omega_v,\omega_u)+g'(x_v(t);\omega_v),\\
\partial^{2}_{vv}H_n(\vec{x}(t))=&\sum_{u;u\sim v}\xi_{vu}f''(x_v(t)-x_u(t);\omega_v,\omega_u)\\
&+\sum_{u;u\sim v}\xi_{uv}f''(x_u(t)-x_v(t);\omega_v,\omega_u)+g''(x_v(t);\omega_v).\\
\end{split}
\end{equation}
And we immediately obtain from \eqref{def:Fi} that
\begin{equation}
\begin{split}
\sum_{v=1}^n\int_0^T\big(\partial_vH_n(\vec{x}(t))\big)^2 \d t = & n\int F^{(2)} \d U(G_n^{x}), \\
\sum_{v=1}^n \int_{0}^{T} \partial^{2}_{vv}H_n(\vec{x}(t)) \d t= & n \int F^{(3)} \d U(G_n^{x}). 
\end{split}
\end{equation}
Combining the equations above concludes the proof.
\end{proof}

\appendix

\section{Properties of BC-entropy}

\par In this section we collect properties of the BC entropy that are used throughout the text. The first result below gives conditions that imply infinite BC entropy. 

\begin{theorem}[\cite{da}, Theorem~1]\label{theorem1}
Let $\mu\in \mathcal{P}(\bar{\mathcal{G}}_*)$ with $0<\deg(\mu)<\infty$, and assume that one of the following conditions holds
\begin{enumerate}
\item $\mu$ is not unimodular.
\item $\mu$ is not supported on $\bar{\mathcal{T}}_*.$
\item either there exist $x,\,x'\in \Xi$ such that $\deg_{x,x'}(\mu)\neq d_{x,x'}$ or there exists $\theta \in \Theta$ such that $\Pi_\theta(\mu)\neq q_\theta$.
\end{enumerate}
Then, for any choice of $(\vec{m}^{(n)},\vec u^{(n)})$ adapted to $(\vec{d},Q)$, we have
\begin{equation*}
\overline \Sigma_{\vec d,Q}(\mu)\vert_{\vec m^{(n)},\vec u^{(n)}}=-\infty.
\end{equation*}
\end{theorem}

The entropy $\Sigma_{\vec{d},Q}$ can be calculated via suitable sequences of approximations, which rely on the generalized configuration model introduced by Bordenave and Caputo~\cite{bc}.

Fix $h \geq 0$ and recall Definition~\ref{def:admissible} of an admissible probability $\PP \in \mathcal{P}(\bar{\mathcal{T}}^{h}_{*})$. If $\PP$ is admissible and $d := \EE(\deg_T (o)) > 0$, let $\pi_\PP$ denote the probability distribution on $(\Xi\times\bar{\mathcal{T}}^{h-1}_*)^{2}$ defined as
\begin{equation}\label{eq:pi_P}
\pi_\PP (t,t' ) := \frac{e_\PP (t,t' )}{d}.
\end{equation}
To see that $\pi_\PP$ is a probability measure, it suffices to observe that, for each $[T,o] \in\bar{\mathcal{T}}_*$, we have
\begin{equation*}
\deg_T (o) = \sum_{t,t'\in \Xi\times\bar{\mathcal{T}}^{h-1}_*} E_h (t,t')(T,o) \quad \text{ and } \quad d = \sum_{t,t' \in \Xi\times\bar{\mathcal{T}}^{h-1}_*}e_\PP (t,t').
\end{equation*}

For $h \geq 1$ and admissible $\PP \in \mathcal{P}(\bar{\mathcal{T}}^h_*)$ with $H(P) < \infty$ and $\EE(\deg_T (o)) > 0$, define the truncated entropy
\begin{equation}\label{eq:J_h}
J_h (\PP) := -s(d) + H(\PP) -\frac{d}{2}H(\pi_\PP )-\sum_{t,t' \in \Xi\times\bar{\mathcal{T}}^{h-1}_*}\mathbb{E}_\PP \big( \log E_h (t,t' )! \big),
\end{equation}
where $d := \mathbb{E}_\PP(\deg_T (o))$ is the average degree at the root and $s(d) = d/2-d/2\log d$. Observe that $H(P)<\infty$ and thus $J_h (\PP) \in [-\infty,\infty)$ is well-defined.

In~\cite{da}, the authors prove the following result, that relates the truncated entropy to the BC entropy of a distribution.
\begin{proposition}[\cite{da}, Theorem 3, Item 2]\label{prop:truncated_entropy}
Let $\PP \in \mathcal{P}_u(\bar{\mathcal{T}}_{*})$ be a probability measure with $\mathbb{E}_\PP(\deg_T(o)\log \deg_T(o))<\infty$. Assume that, for each $h \geq 1$, the measure $\mu_h$ is admissible and $H(\PP_h)<\infty$. Under these hypotheses, it holds that 
\begin{equation*}
\Sigma_{\vec{d},Q}(\PP)=\lim_{h \to \infty} J_h(\PP_h).
\end{equation*}
\end{proposition}

Motivated by the proposition above, we define the notion of strong admissibility.
\begin{definition}
For $h \geq 1$, the probability distribution $\PP \in \mathcal{P}(\bar{\mathcal{T}}_*^h)$ is strongly admissible if $\PP$ is admissible, $\mathbb{E}_\PP [\deg_T (o) \log \deg_T (o)] < \infty$, and $H(\PP) < \infty$.
Let $\mathcal{P}_h$ denote the set of strongly admissible probability distributions $\PP \in \mathcal{P}(\bar{\mathcal{T}}_*^h)$.
\end{definition}

Combining Propositions~5 and~6 of \cite{da} yields the following result.
\begin{proposition}\label{prop:entropyUGW}
Assume that $\PP \in \mathcal{P}_h$.  Then
\begin{equation*}
\Sigma_{\vec{\deg}(\PP),\Pi(\PP)}({\rm UGWT}_h(\PP))=J_h(\PP).
\end{equation*}
\end{proposition}

\section{Compactness in $\bar{\mathcal{G}}_{*}$ and $\mathcal{P}(\bar{\mathcal{G}}_{*})$}\label{app:compacity}

\par In this subsection we provide constructions for families of compact sets on $\bar{\mathcal{G}}_*$ for the case when the spaces of marks $\Theta$ and $\Xi$ are finite.

Since the metrics spaces of marks are finite, compactness of a subset of $\bar{\mathcal{G}}_{*}$ can be deduced from the compactness of the collection of underlying graphs in the space $\mathcal{G}_{*}$. 
Recall that, for a marked graph $G \in \bar{\mathcal{G}}_{*}$, we denote by  $\pi_g(G) \in \mathcal{G}_{*}$ its underlying graph.
\begin{lemma}\label{le:compacts}
Let $K$ be a compact in $\mathcal{G}_{*}$. The set
\begin{equation}
\bar{K} = \{G \in \bar{\mathcal{G}}_{*}: \pi_g(G) \in K\}
\end{equation}
is compact in $\bar{\mathcal{G}}_{*}$.
\end{lemma}
\begin{proof}
Since $\bar{\mathcal{G}}_{*}$ is a metric space, it suffices to verify that $\bar{K}$ is sequentially compact. Consider then a sequence of marked graphs $(G_{k})_{k \in \N}$ in $\bar{K}$ and, by restricting to a subsequence if necessary, assume that $\pi_g(G_{k})$ converges to some rooted graph $g \in \mathcal{G}_{*}$. This implies that, for each $h \geq 1$, $(\pi_g(G_{k}))_{h}$ is eventually equal to $g_{h}$. Write $\tau_{k}$ and $\xi_{k}$ for the vertex- and edge-marks of $G_{k}$, and notice that, for large enough $k$, we can write $(G_{k})_{h}=(g_{h}, \tau_{k,h}, \xi_{k,h})$, where $\tau_{k,h}$ and $\xi_{k,h}$ are the restriction of $\tau_{k}$ and $\xi_{k}$ to $(\pi_g(G_{k}))_{h}$. By using a diagonal argument, we can further restrict our sequence to ensure that, for each $h \geq 1$, the marks on each vertex and edge of $(\pi_g(G_{k}))_{h}$ converge since they live in finite sets. Furthermore, we also have a compatibility structure: if $x \in V(g_{h})$, then the limit of $\tau_{k,h}(x)$ coincides with the limit of $\tau_{k,h+\ell}(x)$, for all $\ell \geq 1$. A similar compatibility holds for edge-marks. This allows us to define vertex- and edge-marks $\tau$ and $\xi$ for $g$, respectively, and obtain a limiting marked graph $\bar{g}=(g, \tau, \xi)$. This implies that $G_{k}$ has a subsequence that converges to $\bar{g}$ and concludes the proof.
\end{proof}

The lemma above reduces the problem of finding compact sets in $\bar{\mathcal{G}}_{*}$ to understanding compacts on $\mathcal{G}_{*}$. The following result, which we state without proof, is useful in this task.
\begin{lemma}[\cite{bc}, Lemma~2.1]\label{le:compactloc}
Fix $t_0 \geq 0$ and $\varphi : \N \to \R_+ $ a non-negative function. For a rooted graph $g\in \mathcal{G}_*$, denote by $e(g_t)$ be the number of edges within distance at most $t$ from the root. Then
\begin{equation*}
K = \big\{g \in \mathcal{G}_* : \, e(g_ t) \leq \varphi (t), \text{ for all } t\, \geq t_0  \big\}
\end{equation*}
is a compact subset of $\mathcal{G}_*$ for the local topology. 
\end{lemma}

\par The advantage the previous lemma provides is that we can now build compact sets by controlling only the growth of the graphs. In this way, we can easily extend~\cite[Lemma~2.3]{bc}. The next result, whose proof follows the same steps from~\cite[Lemma~2.3]{bc} and a straightforward application of Lemma~\ref{le:compacts}, gives us the existence of compacts sets of probability measures $\bar\Pi$ and conditions on $G$ that guarantees $U( G)\in\bar \Pi$.
\begin{lemma}\label{le:tightLWC}
Let $\delta : [0,1] \to \R_+$ be a continuous increasing function such that $\delta(0) = 0$. There exists a compact set $\bar{\Pi} =\bar{\Pi} (\delta) \subset \mathcal{P} (\bar{\mathcal{G}}_{*})$ such that, if a finite marked graph $G$ satisfies 
\begin{equation}\label{eq:tightLWC}
\sum_{v\in S}\deg_{ G}(v) \leq |V| \delta \bigg(\frac{  |S| } { |V | } \bigg), \text{ for all } S\subset V,
\end{equation}
then  $U( G) \in \bar \Pi$.
\end{lemma}

\section{Mixture}\label{appendix:mixture_and_LDP}

\par Let $\Gamma$ and $\mathfrak{X}$ be Polish spaces and consider a sequence $(\mu^n)_{n \in \N} \in \mathcal{P}(\Gamma)$ of probability measures with respective supports in $\Gamma_n \subset \Gamma$. Consider the following $\mu$-assumptions:
\begin{enumerate}
    \item $(\mu^n)_{n \in \N}$ satisfies a large deviation principle with rate function $\psi$.

    \item $(\mu^n)_{n \in \N}$ is exponentially tight.  
\end{enumerate}

For each $\gamma \in \Gamma_n$, let $\PP_\gamma^n \in \mathcal{P}(\mathfrak{X})$ such that
$\gamma \mapsto \PP_\gamma^n(A)$ is measurable on $\Gamma_{n}$, for each $A \subset X$ measurable. 

Define 
\begin{equation*}
\widetilde{\Gamma}=\{\gamma \in \Gamma: \text{ there exist } \gamma_{n} \in \Gamma_n \text{ such that } \gamma_{n} \to \gamma\}.
\end{equation*}
Consider the collection of $\Gamma$-assumptions:
\begin{enumerate}
\item $\widetilde{\Gamma}\neq \emptyset$.

\item For any $\gamma \in \widetilde{\Gamma}$ and $\gamma_{n} \in \Gamma_n$ such that $\gamma_{n} \to \gamma$, the sequence of probabilities $\big(\PP^n_{\gamma_{n}}\big)_{n \in \N}$ is exponentially tight and satisfies a large deviation principle with rate function $\lambda_\gamma$.
\end{enumerate}

\begin{theorem}[Biggins~\cite{biggins}]\label{t:mixture}
Let $(\mu_n)_{n \in \N}$ be a sequence satisfying the $\mu$-assumptions. For each $n \in \N$, consider a family of probability measures $(\PP^{n}_{\gamma})_{\gamma \in \Gamma_{n}}$, and suppose the $\Gamma$-assumptions hold. Then the sequence $(\PP^n)_{n \in \N}$ defined by
\begin{equation} \label{eq:defPmix}
\PP^n(A)=\int_{\Gamma_n} \PP^n_\gamma(A){\rm d} \mu^n(\gamma), \text{ for any measurable set } A \subset \mathfrak{X},
\end{equation}
satisfies a large deviation principle with a good rate function given by
\begin{equation}
\lambda(x)=\inf\{\lambda_\gamma(x)+\psi(\gamma):\gamma \in \Gamma\}.
\end{equation}
\end{theorem}

\section{The approximation lemma}\label{app:approximation_lemma}

\par Roughly speaking, Lemma~\ref{lema:Kurtz} below allows us to derive large deviations for sequences that can be well approximated by collections of sequences that already satisfy large deviation principles. The proof of this lemma is an adaptation of~\cite[Lemma 3.14]{feng_kurtz} for dependent parameters, and we omit it here.
\begin{lemma}\label{lema:Kurtz}
Fix a Polish space $\big( \mathfrak{X}, \d \big)$ and, for each $n$, let $X_n$ be an $\mathfrak{X}$-valued random variable. Suppose that, for each $ k> 0$, there exist a sequence $\{X_n^{k}\}_{n \in \N}$ of $\mathfrak{X}$-valued random variables which satisfies a large deviation principle with good rate function $I^{k}$. Assume furthermore that there exists a function $a: \N \to \R$ with $\lim_{k \to \infty} a(k) = \infty$ such that
\begin{equation}\label{eq:hipothesis_kurtz}
\limsup_{n \to \infty} \frac{1}{n} \log P \left(\d (X_n, X_n^{k}) > \frac{1}{k} \right) \leq -a(k).
\end{equation}
Then $\{X_n\}_{n \in \N}$ satisfies the large deviation principle with good rate function given by
\begin{equation}
I(x) = \limsup_{\delta \to 0} \limsup_{k \to 0} \inf_{y \in B_{\delta}(x)} I^{k}(y) = \liminf_{\delta \to 0} \liminf_{k \to 0} \inf_{y \in B_{\delta}(x)} I^{k}(y).
\end{equation}
\end{lemma}

\section{Additional lemmas}

\par This section contains two standard lemmas that are used in the proofs of Theorems~\ref{t:ldp_uniformg_general} and~\ref{t:ldp_erg_general}.
  The first lemma provides an easy upper bound for deviation of averages of Bernoulli random variables.
\begin{lemma}\label{lemma:leiBinomial}
Let $X_1, \cdots, X_m$ be i.i.d.\ random variables with law $\PP$. For any $a>0$ and measurable set $B$, we have
\begin{equation*}
\mathbb{P} \left( \sum_{i=1}^m 1_{\{X_i \in B\}} > am \right) \leq  \left( \frac{e\PP(B)}{a}\right)^{\lceil am \rceil}.
\end{equation*}
\end{lemma}

\begin{proof}
  Throughout this proof, we denote by $\mathcal{P}_{k}$ the collection of subsets of $[m]$ with size $k$.
  Notice that
\begin{align*}
    \mathbb{P} \left( \sum_{i=1}^m 1_{\{ X_i \in B \}} > am \right) &= \mathbb{P} \left( \text{there exists } S \in \mathcal{P}_{ \lceil am \rceil } : X_i \in B, \text{ for all } i \in S \right)\\
    &\leq \sum_{ S \in \mathcal{P}_{ \lceil am \rceil } } \PP(B)^{\lceil am \rceil} \leq \binom{m}{\lceil am \rceil} \PP(B)^{\lceil am \rceil} .
\end{align*}
Combining this with the standard bound $\binom{m}{k} \leq \left( \frac{em}{k}\right)^k$ concludes the proof.
\end{proof}

Second, we provide a lemma that bounds the number of edges in the sparse Erd\H{o}s-R\'enyi random graph.
\begin{lemma}\label{lemma:number_of_edges_ER}
Let $G_{n}$ be distributed according to $G\big( n, \frac{d}{n} \big)$. Given $z > 0$, there exists $a>0$ such that, for all $n \in \N$,
\begin{equation}
\P\big( |E_{n}| \geq an \big) \leq e^{-zn}.
\end{equation}
\end{lemma}

\begin{proof}
  Notice that $|E_{n}| \sim \binomial\left(\frac{n(n-1)}{2}, \frac{d}{n}\right)$.
  From this, we obtain
\begin{equation*}
\begin{split}
\P \big( |E_{n}| \geq an \big) & \leq e^{-an}\E\big[ e^{|E_{n}|} \big] \leq e^{-an}\left(e\frac{d}{n}+1-\frac{d}{n} \right)^{n^{2}} \\
& \leq e^{-an}e^{(e-1)\frac{d}{n}n^{2}} \leq e^{-zn},
\end{split}
\end{equation*}
if $a$ is chosen large enough, concluding the proof.
\end{proof}

Finally, we prove a lemma that regards the amount of vertices whose neighborhoods of a given size are large.
\begin{lemma}\label{lemma:caputoadaptado}
Let $G_{n}$ denote the uniform random graph with $m_{n}$ edges, where $\tfrac{m_{n}}{n} \to \tfrac{d}{2}$, or the Erd\"{o}s-R\'{e}nyi random graph $G \big( n, \tfrac{d}{n} \big)$.
For any $\varepsilon>0$
\begin{equation}\label{eq:bad_event}
\limsup_{h \to \infty} \limsup_{S \to \infty} \limsup_{n \to \infty} \frac{1}{n} \log \mathbb{P}\left(  \frac{1}{n} \sum_{v \in V_n} \mathbf{1}_{\{|[G_n,v]_h| > S\}}  > \varepsilon \right) = -\infty.
\end{equation}
\end{lemma}

\begin{proof}
  This proof is a consequence of~\cite[Lemmas 2.3 and 6.2]{bc}.
\end{proof}

\noindent \textbf{Data Availability Statement.} Data sharing not applicable to this article as no datasets were generated or analysed during the current study.

\noindent\textbf{Competing Interests.}  The authors declare that they have no financial or non-financial competing interests related to the work submitted.

\bibliographystyle{plain}
\bibliography{mybib}

\end{document}